\documentclass[12pt,a4paper]{article}
\usepackage{amssymb}
\usepackage{amsmath, amscd}
\usepackage{amsthm}
\usepackage{mathtools}
\usepackage{mathrsfs}
\usepackage{booktabs}
\usepackage{perpage}
\usepackage[all,cmtip]{xy}
\usepackage{setspace}
\usepackage{amsbsy}
\usepackage[T1]{fontenc}
\usepackage{verbatim}
\usepackage{mathdots}
\usepackage{mathrsfs}
\usepackage{ifthen}
\usepackage{blkarray}
\usepackage{multirow}
\usepackage{hyperref}
\usepackage{enumerate}
\usepackage[dvipsnames]{xcolor}
\usepackage{tikz}
\usepackage{fullpage}
\usepackage{colonequals}
\usepackage{braket}
\usepackage{lipsum}
\usepackage{float}

\usepackage{graphicx}

\makeatletter
\newcommand{\sigmaop}[1]{\mathop{\mathpalette\@sigmaop{#1}}\slimits@}
\newcommand{\@sigmaop}[2]{%
  \vphantom{\sum}%
  \sbox\z@{$\m@th#1\sum$}%
  \dimen@=\ht\z@ \advance\dimen@\dp\z@
  \dimen\tw@=\wd\z@
  \ifx#1\displaystyle\dimen@=.9\dimen@\fi
  \ooalign{%
    \hidewidth
    $\vcenter{\hbox{$\m@th#1#2$\kern.3\dimen\tw@}%
     \ifx#1\scriptstyle\kern-.25ex\fi}$\hidewidth\cr
    $\vcenter{\hbox{%
      \resizebox{!}{\dimen@}{$\m@th\boxtimes$}%
    }\ifx#1\scriptstyle\kern-.25ex\fi}$\cr
  }%
}
\makeatother

\mathtoolsset{showonlyrefs=true}

\numberwithin{equation}{subsection}

\newtheorem{theorem}{Theorem}[subsection]
\newtheorem{lemma}[theorem]{Lemma}

\newtheorem{conjecture}[theorem]{Conjecture}

\newtheorem{corollary}[theorem]{Corollary}
\newtheorem{definition}[theorem]{Definition}

\newtheorem{proposition}[theorem]{Proposition}

\newtheorem*{thm1}{Theorem 1}
\newtheorem*{thm2}{Theorem 2}
\newtheorem*{thm3}{Theorem 3}

\newtheorem*{conj1}{Conjecture 1}

\theoremstyle{remark}

\newtheorem{exa}[theorem]{Example}

\newcommand{\GZip}{\mathop{\text{$G$-{\tt Zip}}}\nolimits}

\newcommand{\GiZip}{\mathop{\text{$G_i$-{\tt Zip}}}\nolimits}

\newcommand{\GoneZip}{\mathop{\text{$G_1$-{\tt Zip}}}\nolimits}

\newcommand{\GtwoZip}{\mathop{\text{$G_2$-{\tt Zip}}}\nolimits}

\newcommand{\GpZip}{\mathop{\text{$G'$-{\tt Zip}}}\nolimits}
\newcommand{\GF}{\mathop{\text{$G$-{\tt ZipFlag}}}\nolimits}

\newskip\procskipamount
\newskip\interskipamount
\newskip\refskipamount
\newcommand{\procskip}{\vskip\procskipamount}
\newcommand{\interskip}{\vskip\interskipamount}
\newcommand{\refskip}{\vskip\refskipamount}

\newcommand{\procbreak}{\par
   \ifdim\lastskip<\procskipamount\removelastskip
   \penalty-100
   \procskip\fi
   \noindent\ignorespaces}

\newcommand{\titlebreak}{\par%
\ifdim\lastskip<\interskipamount\removelastskip%
\penalty10000%
\interskip\fi%
\noindent}%

\newcommand{\interbreak}{\par%
\ifdim\lastskip<\interskipamount\removelastskip%
\penalty-100%
\interskip\fi%
\noindent\ignorespaces}%

\newcommand{\refbreak}{\par%
\ifdim\lastskip<\refskipamount\removelastskip%
\penalty-100%
\refskip\fi%
\noindent\ignorespaces}%

\newcounter{listcounter}
\newcounter{deflistcounter}
\newcounter{equivcounter}

\newskip{\itemsepamount}
\newskip{\topsepamount}
\newenvironment{assertionlist}{%
  \begin{list}
    {\upshape (\arabic{listcounter})}
    {\setlength{\leftmargin}{18pt}
     \setlength{\rightmargin}{0pt}
     \setlength{\itemindent}{0pt}
     \setlength{\labelsep}{5pt}
     \setlength{\labelwidth}{13pt}
     \setlength{\listparindent}{\parindent}
     \setlength{\parsep}{0pt}
     \setlength{\itemsep}{\itemsepamount}
     \setlength{\topsep}{\topsepamount}
     \usecounter{listcounter}}}
  {\end{list}}

\newenvironment{definitionlist}{%
  \begin{list}
    {\upshape (\alph{deflistcounter})}
    {\setlength{\leftmargin}{18pt}
     \setlength{\rightmargin}{0pt}
     \setlength{\itemindent}{0pt}
     \setlength{\labelsep}{5pt}
     \setlength{\labelwidth}{13pt}
     \setlength{\listparindent}{\parindent}
     \setlength{\parsep}{0pt}
     \setlength{\itemsep}{\itemsepamount}
     \setlength{\topsep}{\topsepamount}
     \usecounter{deflistcounter}}}
  {\end{list}}

\newenvironment{Alist}{%
  \begin{list}
    {\upshape (\Alph{deflistcounter})}
    {\setlength{\leftmargin}{18pt}
     \setlength{\rightmargin}{0pt}
     \setlength{\itemindent}{0pt}
     \setlength{\labelsep}{5pt}
     \setlength{\labelwidth}{13pt}
     \setlength{\listparindent}{\parindent}
     \setlength{\parsep}{0pt}
     \setlength{\itemsep}{\itemsepamount}
     \setlength{\topsep}{\topsepamount}
     \usecounter{deflistcounter}}}
  {\end{list}}

\newenvironment{equivlist}{%
  \begin{list}
    {\upshape (\roman{equivcounter})}
    {\setlength{\leftmargin}{18pt}
     \setlength{\rightmargin}{0pt}
     \setlength{\itemindent}{0pt}
     \setlength{\labelsep}{5pt}
     \setlength{\labelwidth}{13pt}
     \setlength{\listparindent}{\parindent}
     \setlength{\parsep}{0pt}
     \setlength{\itemsep}{\itemsepamount}
     \setlength{\topsep}{\topsepamount}
     \usecounter{equivcounter}}}
  {\end{list}}

\newenvironment{bulletlist}{%
  \begin{list}
    {\upshape \textbullet}
    {\setlength{\leftmargin}{18pt}
     \setlength{\rightmargin}{0pt}
     \setlength{\itemindent}{0pt}
     \setlength{\labelsep}{6pt}
     \setlength{\labelwidth}{12pt}
     \setlength{\listparindent}{\parindent}
     \setlength{\parsep}{0pt}
     \setlength{\itemsep}{\itemsepamount}
     \setlength{\topsep}{\topsepamount}}}
  {\end{list}}

% ...

\newcommand{\Bcal}{{\mathcal B}}
\newcommand{\Ccal}{{\mathcal C}}

\newcommand{\Ecal}{{\mathcal E}}
\newcommand{\Fcal}{{\mathcal F}}
\newcommand{\Gcal}{{\mathcal G}}

\newcommand{\Lcal}{{\mathcal L}}

\newcommand{\Ocal}{{\mathcal O}}

\newcommand{\Scal}{{\mathcal S}}

\newcommand{\Ucal}{{\mathcal U}}
\newcommand{\Vcal}{{\mathcal V}}

\newcommand{\Zcal}{{\mathcal Z}}

\renewcommand{\AA}{\mathbb{A}}

\newcommand{\CC}{\mathbb{C}}

\newcommand{\FF}{\mathbb{F}}
\newcommand{\GG}{\mathbb{G}}

\newcommand{\NN}{\mathbb{N}}

\newcommand{\QQ}{\mathbb{Q}}
\newcommand{\RR}{\mathbb{R}}
\renewcommand{\SS}{\mathbb{S}}

\newcommand{\ZZ}{\mathbb{Z}}

\DeclareMathOperator{\Pic}{Pic}

\newcommand{\Sscr}{{\mathscr S}}

\newcommand{\sgn}{\mbox{sgn}}

\DeclareMathOperator{\Gal}{Gal}

\DeclareMathOperator{\rank}{rank}
\DeclareMathOperator{\Span}{Span}

\DeclareMathOperator{\pr}{pr}

\DeclareMathOperator{\Ker}{Ker}

\DeclareMathOperator{\Sh}{Sh}

\DeclareMathOperator{\Sym}{Sym}

\DeclareMathOperator{\SL}{SL}
\DeclareMathOperator{\GL}{GL}
\DeclareMathOperator{\PGL}{PGL}
\DeclareMathOperator{\GSp}{GSp}

\DeclareMathOperator{\Sp}{Sp}
\DeclareMathOperator{\U}{U}
\DeclareMathOperator{\GU}{GU}

\newcommand{\id}{{\rm Id}}

\newcommand{\loccit}{{\em loc.\ cit. }}
\newcommand{\loccitn}{{\em loc.\ cit.}}

\newcommand{\diag}{{\rm diag}}

\newcommand{\fil}{{\rm Fil}}

\DeclareMathOperator{\Std}{Std}

\renewcommand{\div}{{\rm div}}

\DeclareMathOperator{\ind}{ind}

\DeclareMathOperator{\mult}{mult}

\DeclareMathOperator{\Norm}{Norm}

\DeclareMathOperator{\deter}{det}

\DeclareMathOperator{\Res}{Res}
\DeclareMathOperator{\Flag}{Flag}

\DeclareMathOperator{\Frac}{Frac}

\DeclareMathOperator{\Eff}{Eff}

\DeclareMathOperator{\Cox}{Cox}

\newcommand{\relmiddle}[1]{\mathrel{}\middle#1\mathrel{}}

\newcommand{\JS}[1]{{\color{ForestGreen}  [#1]}}

\def\numero{%
  \leavevmode
  \hbox{%
    $\rm N^{\mkern0.8mu\underline{\mkern-0.8mu o\mkern-0.8mu}\mkern0.8mu}$%
  }%
}

\newcommand{\xdasharrow}[2][->]{
\tikz[baseline=-\the\dimexpr\fontdimen22\textfont2\relax]{
\node[anchor=south,font=\scriptsize, inner ysep=1.5pt,outer xsep=2.2pt](x){#2};
\draw[shorten <=3.4pt,shorten >=3.4pt,dashed,#1](x.south west)--(x.south east);
}
}

\newcommand{\Ha}{\mathsf{Ha}}
\newcommand{\ha}{\mathsf{ha}}
\newcommand{\pha}{\mathsf{pHa}}
\newcommand{\zip}{\mathsf{zip}}
\newcommand{\GS}{\mathsf{GS}}

\newcommand{\flag}{\mathsf{flag}}

\newcommand{\low}{\mathsf{low}}

\begin{document}

\title{The stack of $G$-zips is a Mori dream space} 

\author{Jean-Stefan Koskivirta}

\date{}

\maketitle

\begin{abstract}
We first extend previous results of the author with T. Wedhorn and W. Goldring regarding the existence of $\mu$-ordinary Hasse invariants for Hodge-type Shimura varieties to other automorphic line bundles. We also determine exactly which line bundles admit nonzero sections on the stack of $G$-zips of Pink--Wedhorn--Ziegler. Then, we define and study the Cox ring of the stack of $G$-zips and show that it is always finitely generated. Finally, beyond the case of line bundles, we define a ring of vector-valued automorphic forms on the stack of $G$-zips and study its properties. We prove that it is finitely generated in certain cases.
\end{abstract}

\section{Introduction}

The broad topic of this paper is automorphic forms in characteristic $p$. Similarly to classical automorphic forms, they can be defined as global sections of certain vector bundles on the special fiber of Shimura varieties. For a reductive group $G$ over $\FF_p$, the stack of $G$-zips of Pink--Wedhorn--Ziegler affords a group-theoretical analogue of Shimura varieties. Specifically, let $G$ be a connected, reductive group over $\FF_p$ and $\mu\colon \GG_{\mathrm{m},k}\to G_{k}$ a cocharacter of $G_{k}$, where $k=\overline{\FF}_p$ is a fixed algebraic closure of $\FF_p$. Attached to the pair $(G,\mu)$, there is an associated stack of $G$-zips, denoted by $\GZip^\mu$, as defined in \cite{Pink-Wedhorn-Ziegler-zip-data, Pink-Wedhorn-Ziegler-F-Zips-additional-structure}. It appears in several aspects of the theory of Shimura varieties in characteristic $p$ and is conjectured to share many geometric properties with them. This approach makes it possible to study questions related to singularities, cohomology vanishing, etc. from a group-theoretical point of view by studying similar questions on the stack $\GZip^\mu$. In this paper, we will prove that $\GZip^\mu$ is a Mori dream space.

The stack of $G$-zips can be defined as a classifying stack of certain torsors. It also admits an expression as a quotient stack (see section \ref{zip-def} for details). It relates to the theory of Shimura varieties as follows. Let $S_K$ be the special fiber at a place of good reduction of a Hodge-type Shimura variety, whose attached reductive group (over $\FF_p$) is $G$. Then, Zhang has constructed a smooth morphism of stacks
\begin{equation}
    \zeta\colon S_K\to \GZip^\mu
\end{equation}
that can be shown to be surjective. The fibers of this map are called the Ekedahl--Oort strata of $S_K$. There is a unique open point in $\GZip^\mu$, and the corresponding stratum is called the $\mu$-ordinary locus of $S_K$. We think of the stack $\GZip^\mu$ as a group-theoretical version of $S_K$. Beside the aforementioned stratification, it also shares another important aspect of Shimura varieties, as it carries a natural family of vector bundles (termed automorphic vector bundles). To explain this, recall that one can naturally attach to the pair $(G,\mu)$ a zip datum $\Zcal_\mu=(G,P,Q,L,M)$ consisting of two parabolic subgroups $P,Q$ with respective Levi subgroups $L,M$. Then, to any algebraic representation $\rho\colon P\to \GL(V)$, one attaches a vector bundle $\Vcal(\rho)$ on $\GZip^\mu$. In the special case when $\rho$ is a character $\lambda\in X^*(P)=X^*(L)$, the attached vector bundle is a line bundle, that we denote by $\Lcal(\lambda)$.

In \cite{Koskivirta-Wedhorn-Hasse}, the author and T. Wedhorn constructed $\mu$-ordinary Hasse invariants for general pairs $(G,\mu)$ (later extended by W. Goldring and the author in \cite{Goldring-Koskivirta-Strata-Hasse}). Fix an appropriate Borel pair $(B,T)$ (we impose that $B\subset P$, $T\subset L$ and that $(B,T)$ be defined over $\FF_p$), and let $\Delta$ denote the simple roots. Write $I\subset \Delta$ for the set of simple roots of $L$ and set $\Delta^P\colonequals \Delta \setminus I$. In \loccitn, the author and T. Wedhorn showed the following: If $\lambda\in X^*(L)$ satisfies the condition
\begin{equation}\label{cond-wed}
\langle \lambda, \alpha^\vee \rangle < 0 \quad \textrm{for all} \ \alpha\in \Delta^P,
\end{equation}
then there exists $N\geq 1$ and a section
\begin{equation}
    \Ha_\mu \in H^0(\GZip^\mu,\Lcal(\lambda)^{\otimes N})
\end{equation}
whose non-vanishing locus is exactly the open stratum of $\GZip^\mu$ (i.e. the $\mu$-ordinary locus). Such a section is called a $\mu$-ordinary Hasse invariant. It is a group-theoretical generalization of the classical Hasse invariant. Pulling back this section via the map $\zeta$, one obtains similar invariants over the Shimura variety $S_K$. Our first result in the present paper is to generalize the construction of $\mu$-ordinary Hasse invariants and give a more precise statement regarding their existence. Furthermore, we answer the question of when the space $H^0(\GZip^\mu,\Lcal(\lambda))$ is nonzero. Define a map 
\begin{equation}
\wp_*\colon X_*(T)_\QQ\to X_*(T)_\QQ, \quad \wp_* \colon \delta \mapsto \delta-p\sigma(\delta)
\end{equation}
where $\sigma\in \Gal(k/\FF_p)$ is the $p$-power Frobenius homomorphism. For each root $\alpha\in \Delta^P$, define a quasi-cocharacter $\delta_\alpha\in X_*(T)_\QQ$ as the preimage $\delta_\alpha \colonequals \wp^{-1}(\alpha^\vee)$. We prove:

\begin{thm1}[Theorem \ref{thm-mu}, Proposition \ref{prop-line}]
Let $\lambda\in X^*(L)$ be a character.
\begin{assertionlist}
    \item There exists a $\mu$-ordinary Hasse invariant $h\in H^0(\GZip^\mu,\Lcal(\lambda)^{\otimes m})$ (for some $m\geq 1)$ if and only if $\lambda$ satisfies $\langle \lambda, \delta_\alpha \rangle > 0$ for all $\alpha\in \Delta^P$.
    \item The space $H^0(\GZip^\mu,\Lcal(\lambda)^{\otimes m})$ is nonzero (for some $m\geq 1$) if and only if $\lambda$ satisfies $\langle \lambda, \delta_\alpha \rangle \geq 0$ for all $\alpha\in \Delta^P$.
\end{assertionlist}
\end{thm1}

The purpose of the paper \cite{Koskivirta-Wedhorn-Hasse} was to construct such invariants as a section of a power of the Hodge line bundle $\omega$, which exists naturally on $S_K$. The character $\lambda_\omega$ corresponding to $\omega$ satisfies the stronger conditions \eqref{cond-wed}, so the result of \loccit was sufficient for its purpose. Similar sections were constructed by Goldring--Nicole for unitary Shimura varieties (\cite{Goldring-Nicole-mu-Hasse}) and W. Goldring and the author in \cite{Goldring-Koskivirta-Strata-Hasse}. However, it can be useful in general to construct such invariants for other automorphic line bundles $\Lcal(\lambda)$, beyond the case of the Hodge line bundle $\omega$. Conditions (ii) in Theorem 1 are in general much less restrictive than \eqref{cond-wed}.

The next purpose of this article is to study the spaces $H^0(\GZip^\mu,\Lcal(\lambda))$ "all at once". For this, the standard construction is to form the direct sum:
\begin{equation*}
   \Cox(\GZip^\mu) \colonequals \bigoplus_{\lambda\in X^*(L)} H^0(\GZip^\mu, \Lcal(\lambda)).
\end{equation*}
This object inherits a natural structure of $k$-algebra, and we call it the Cox ring of $\GZip^\mu$, by analogy with the Cox ring of a projective variety. Recall that a projective $k$-variety $X$ with finitely generated Picard group $\Pic(X)$ is called a Mori dream space if its Cox ring $\Cox(X)$ is a finitely generated $k$-algebra. Examples of such spaces are spherical varieties and smooth Fano varieties. For an arbitrary pair $(G,\mu)$, we show:
\begin{thm2}[Theorem \ref{cox-fin-gen-thm}]
The ring $\Cox(\GZip^\mu)$ is a finitely generated $k$-algebra. In other words, the stack $\GZip^\mu$ is a Mori dream space.
\end{thm2}
We also show some elementary properties of the ring $\Cox(\GZip^\mu)$: it is an integral domain which is integrally closed, and its unit group identifies with the group of nowhere-vanishing functions on $G$ (Proposition \ref{Cox-first}). When $G$ is semisimple and $\Pic(G)=0$, the Cox ring of $\GZip^\mu$ is a polynomial algebra (Corollary \ref{cor-semisimple}).

In the last part of the article, we consider more general, higher rank automorphic vector bundles. Let $\lambda\in X^*(T)$ be a character, and consider the induced $P$-representation $V_I(\lambda)\colonequals \ind_B^P(\lambda)$. Write $\Vcal_I(\lambda)$ for the associated vector bundle on $\GZip^\mu$. In the context of Shimura varieties, the pullback $\zeta^*(\Vcal_I(\lambda))$ coincides with the automorphic vector bundle attached to $\lambda$. Global sections in $H^0(S_K,\Vcal_I(\lambda))$ are called mod $p$ automorphic forms of level $K$ and weight $\lambda$. They play a central role in the mod $p$ Langlands program, as Hecke-invariant forms are conjectured to correspond to mod $p$ Galois representations. Again, we may form the following direct sum:
\begin{equation}
     R_{\zip}\colonequals \bigoplus_{\lambda\in X^*(T)} H^0(\GZip^\mu,\Vcal_{I}(\lambda)).
\end{equation}
The natural map $V_I(\lambda)\otimes_k V_I(\lambda') \to V_I(\lambda+\lambda')$ (for $\lambda,\lambda'\in X^*(T)$) endows $R_{\zip}$ with a structure of a $k$-algebra. We call $R_{\zip}$ the ring of automorphic forms on $\GZip^\mu$. It contains the Cox ring of $\GZip^\mu$ as a subring. One may also interpret $R_{\zip}$ as the Cox ring of the stack of $G$-zip flags (see \ref{sec-zip-flag} for details). We conjectured in \cite[Conjecture 5.1.4]{Koskivirta-automforms-GZip}:
\begin{conj1}
The ring $R_{\zip}$ is finitely generated.
\end{conj1}
In \loccitn, we carried out the case of the group $G=\Sp_{4,\FF_p}$, where we proved that $R_{\zip}$ is a polynomial algebra in three variables. In the present article, we prove general properties of this ring: it is an integrally closed integral domain, whose field of fractions identifies with $k(B\cap M)$ (Proposition \ref{prop-Kzip}). We show that the veracity of Conjecture 1 is unchanged if we modify $G$ by a central subgroup (Proposition \ref{Rprime-vs-R}), or if we take a product over $\FF_p$ of two $\FF_p$-reductive groups (Proposition \ref{prop-Fp-prod}). We also show that it holds in the extremal cases $P=G$ or $P=B$ (Proposition \ref{PBPG}). Finally, we illustrate the case of a general linear group in three dimensions (and its $\FF_p$-forms given by unitary groups) in the theorem below.

\begin{thm3}[Corollary \ref{cor-split-GL3}, Proposition \ref{prop-inert-U}]
Let $G$ be a reductive $\FF_p$-group such that $G_k\simeq \GL_{3,k}$. For any cocharacter $\mu\colon \GG_{\mathrm{m},k}\to G_k$, Conjecture 1 holds true.
\end{thm3}

We end this introduction by explaining the content of each chapter. In chapter 2, we review some basic definitions regarding the stack of $G$-zips, their connection with Shimura varieties, the construction of vector bundles attached to algebraic representations. Chapter 3 is devoted to studying the space of global sections of vector bundles over $\GZip^\mu$, using as a main ingredient the results of our previous paper with N. Imai \cite{Imai-Koskivirta-vector-bundles}. We also prove Theorem 1 regarding $\mu$-ordinary Hasse invariants. In Chapter 4, we introduce the Cox ring of $\GZip^\mu$ and study its properties. We prove that it is finitely generated (Theorem 2). The last chapter deals with the ring $R_{\zip}$ of automorphic forms. We illustrate the case of ($\FF_p$-forms of) $\GL_{3}$ and prove Theorem 3.

\section{\texorpdfstring{The stack of $G$-zips}{}}

\subsection{Definition} \label{zip-def}
Fix an algebraic closure $k=\overline{\FF}_p$ of $\FF_p$. For a characteristic $p$ scheme $X$, we write $X^{(p)}$ for the Frobenius-twist of $X$ and $F_X\colon X\to X^{(p)}$ for the relative Frobenius morphism. Let $G$ be a connected, reductive group over $\FF_p$ and $\mu\colon \GG_{\mathrm{m},k}\to G_k$ a cocharacter. We call the pair $(G,\mu)$ a cocharacter datum. We can attach to $(G,\mu)$ a tuple $\Zcal_\mu=(G,P,Q,L,M)$ called a zip datum. Here, $P$, $Q$ are parabolic subgroups of $G_k$ defined as follows. First let $P_{+}(\mu)$ (resp. $P_-(\mu)$) be the parabolic subgroup whose $k$-points are the elements $g\in G(k)$ such that the map 
\[
\GG_{\mathrm{m},k} \to G_{k}; \  t\mapsto\mu(t)g\mu(t)^{-1} \quad (\textrm{resp. } t\mapsto\mu(t)^{-1}g\mu(t))
\]
extends to a morphism of varieties $\AA_{k}^1\to G_{k}$. We obtain a pair of opposite parabolics $(P_+(\mu),P_{-}(\mu))$ in $G_k$ whose intersection $P_+(\mu)\cap P_-(\mu)=L(\mu)$ is the centralizer of $\mu$ (it is a common Levi subgroup of $P_+(\mu)$ and $P_-(\mu)$). We define $P \colonequals P_-(\mu)$, $Q \colonequals (P_+(\mu))^{(p)}$, $L \colonequals L(\mu)$ and $M \colonequals  L^{(p)}$. Write $\varphi\colon L\to M$ for the Frobenius homomorphism. For an algebraic group $H$, we write $R_{\mathrm{u}}(H)$ for its unipotent radical. Let $\theta^P_L\colon P\to L$ be the projection onto the Levi subgroup $L$, i.e. the map defined by $\theta^P_L(xu)=x$ for all $x\in L$ and $u\in R_{\mathrm{u}}(P)$. Define $\theta^Q_M\colon Q\to M$ similarly. The zip group of $(G,\mu)$ is defined by:
    \begin{equation}\label{zipgroup}
E \colonequals \{(x,y)\in P\times Q \mid  \varphi(\theta^P_L(x))=\theta^Q_M(y)\}.
\end{equation}
We let $E$ act on $G_k$ by the action $(x,y)\cdot g = xgy^{-1}$ for all $(x,y)\in E$, $g\in G_k$. Moonen--Wedhorn (\cite{Moonen-Wedhorn-Discrete-Invariants} and Pink--Wedhorn--Ziegler (\cite{Pink-Wedhorn-Ziegler-zip-data,Pink-Wedhorn-Ziegler-F-Zips-additional-structure}) defined the stack of $G$-zips of type $\mu$. It can be expressed as a moduli stack of certain torsors. We will use the following, equivalent definition as a quotient stack
\begin{equation}
    \GZip^\mu \colonequals \left[ E\backslash G_k \right].
\end{equation}
This stack is a smooth $k$-stack with finite underlying topological space.

\subsection{Group-theoretical notation} \label{sec-gp-not}

Let again $(G,\mu)$ be a cocharacter datum and $\Zcal_\mu=(G,P,Q,L,M)$ be the associated zip datum as in section \ref{zip-def}. We explain notation pertaining to the root datum of $G$, used throughout the paper. We will always make the following assumption: there exists a Borel pair $(B,T)$ defined over $\FF_p$ satisfying $B\subset P$ and $T\subset L$. After changing $\mu$ up to conjugation, this assumption can always be achieved. Let $B^+$ denote the opposite Borel subgroup, i.e. the unique Borel subgroup such that $B\cap B^+=T$.

We let $X^*(T)$ (resp. $X_*(T)$) denote the group of characters (resp. cocharacters) of $T$. Since $T$ is defined over $\FF_p$, these sets are naturally endowed with an action of $\Gal(k/\FF_p)$. Write $W=W(G_k,T)$ for the Weyl group of $G_k$. Similarly, $\Gal(k/\FF_p)$ acts on $W$ and the actions of $\Gal(k/\FF_p)$ and $W$ on $X^*(T)$ and $X_*(T)$ are compatible in a natural sense. We let $\Phi,\Phi^+,\Delta$ denote respectively the set of $T$-roots of $G$, the positive roots and the simple roots. Our convention of positivity differs from some authors: A root $\alpha$ lies in $\Phi^+$ if and only if the $\alpha$-root group $U_{\alpha}$ is contained in $B^+$. For a root $\alpha \in \Phi$, we let $s_\alpha \in W$ be the corresponding reflection. The system $(W,\{s_\alpha \mid \alpha \in \Delta\})$ is a Coxeter system. We write $\ell  \colon W\to \NN$ for the length function. In particular, $\ell(s_\alpha)=1$ for all $\alpha\in \Delta$. Let $w_0\in W$ be the longest element of $W$. Write $\Phi^+_L\subset \Phi^+$ for the positive $T$-roots of $L$, and let $I\colonequals \Delta_L=\Phi^+_L\cap \Delta$ denote the simple roots of $L$. We also define $\Delta^P \colonequals \Delta \setminus I$. Let $W_I\colonequals W(L,T)$ be the Weyl group of $L$. Write $w_{0,I}$ for the longest element in $W_I$. Define ${}^I W$ to be the subset of elements $w\in W$ which are of minimal length in the right coset $W_I w$. The set ${}^I W$ is a set of representatives for the quotient $W_I\backslash W$. The longest element in the set ${}^I W$ is $w_{0,I} w_0$.

We say that a character $\lambda\in X^*(T)$ is dominant if $\langle \lambda,\alpha^\vee \rangle \geq 0$ for all $\alpha \in \Delta$. The set of dominant characters will be denoted by $X_{+}^*(T)$. Similarly, if $\langle \lambda,\alpha^\vee \rangle \geq 0$ for all $\alpha \in I$, we say that $\lambda$ is $I$-dominant (or $L$-dominant), and write $X_{+,I}^*(T)$ for the set of $I$-dominant characters. For each $\alpha\in \Phi$, choose an isomorphism $u_\alpha\colon \GG_{\mathrm{a}}\to U_\alpha$ so that 
$(u_{\alpha})_{\alpha \in \Phi}$ is a realization in the sense of \cite[8.1.4]{Springer-Linear-Algebraic-Groups-book}. In particular, we have 
\begin{equation}\label{eq:phiconj}
 t u_{\alpha}(x)t^{-1}=u_{\alpha}(\alpha(t)x), \quad \forall x\in \GG_{\mathrm{a}},\  \forall t\in T.
\end{equation}

\subsection{\texorpdfstring{Stratification of $\GZip^\mu$}{}}
\label{subsec-zipstrata}

Let $(G,\mu)$ be a cocharacter datum as in \ref{zip-def} and let $\GZip^\mu =[E\backslash G_k]$ be the attached stack of $G$-zips. We recall the parametrization of the points of the underlying topological space of $\GZip^\mu$. They correspond bijectively to the $E$-orbits in $G_k$. By \cite[Proposition 7.1]{Pink-Wedhorn-Ziegler-zip-data}, there are finitely many $E$-orbits in $G_k$. Furthermore, each $E$-orbit is smooth and locally closed in $G_k$, and the Zariski closure of an $E$-orbit is a union of $E$-orbits. For $w\in W$, we fix a representative $\dot{w}\in N_G(T)$, such that $(w_1w_2)^\cdot = \dot{w}_1\dot{w}_2$ whenever $\ell(w_1 w_2)=\ell(w_1)+\ell(w_2)$ (this is possible by choosing a Chevalley system, \cite[ XXIII, \S6]{SGA3}). We set throughout
\begin{equation} \label{def-z}
z\colonequals \sigma(w_{0,I})w_0.    
\end{equation}
For $w\in W$, write
\begin{equation}
    G_w\colonequals E\cdot \dot{w}\dot{z}^{-1}
\end{equation}
for the $E$-orbit of the element $\dot{w}\dot{z}^{-1}$. We will simply write $w$ instead of $\dot{w}$. Define a partial order $\preccurlyeq$ on ${}^I W$ as follows: For $w,w'\in {}^I W$, write $w'\preccurlyeq w$ if there exists $w_1\in W_I$ such that $w'\leq w_1 w \sigma(w_1)^{-1}$ (see \cite[Corollary 6.3]{Pink-Wedhorn-Ziegler-zip-data}).

\begin{theorem}[{\cite[Theorem 7.5]{Pink-Wedhorn-Ziegler-zip-data}}] \label{thm-E-orb-param} \ 
\begin{assertionlist}
\item The map $w\mapsto G_w$ is a bijection ${}^I W \rightarrow \{ \textrm{$E$-orbits in $G_k$} \}$.
\item For $w\in {}^I W$, one has $\dim(G_w)= \ell(w)+\dim(P)$.
\item For $w\in {}^I W$, the Zariski closure of $G_w$ is 
\begin{equation}\label{equ-closure-rel}
\overline{G}_w=\bigsqcup_{w'\in {}^IW,\  w'\preccurlyeq w} G_{w'}.
\end{equation}
\end{assertionlist}
\end{theorem}

In particular, there is a unique open $E$-orbit
$U_\mu\subset G_k$ corresponding via the map $w\mapsto G_w$ to the longest element $w_{0,I}w_0\in {}^I W$. It coincides with the $E$-orbit of the identity element $1\in G_k$. Using the terminology pertaining to Shimura varieties, we call $U_\mu$ the \emph{$\mu$-ordinary stratum} of $G_k$ and the open substack $\Ucal_\mu \colonequals [E\backslash U_\mu]$ the \emph{$\mu$-ordinary locus}. It corresponds to the $\mu$-ordinary locus of the special fiber of Shimura varieties, studied in \cite{Moonen-Serre-Tate} and \cite{Wortmann-mu-ordinary}. See section \ref{shim-sec} below for details. The codimension one $E$-orbits correspond bijectively to elements of ${}^I W$ of length $\ell(w_{0,I}w_0)-1$. Such elements can be written $w_{0,I}s_\alpha w_0$ for $\alpha \in \Delta^P$. We will simply write $Z_{\alpha}$ for the $E$-orbit corresponding to $w_{0,I}s_\alpha w_0$. One sees easily that $Z_\alpha$ is the $E$-orbit of $s_\alpha$.

\subsection{Shimura varieties}\label{shim-sec}

We briefly explain the connection between $G$-zips and  Shimura varieties. Let $(\mathbf{G},\mathbf{X})$ be a Shimura datum of Hodge-type (i.e. it admits an embedding into a Siegel-type Shimura datum). In particular, $\mathbf{G}$ is a connected, reductive group over $\QQ$. Let $K\subset \mathbf{G}(\AA_f)$ be a sufficiently small compact open subgroup and write $\Sh_K(\mathbf{G},\mathbf{X})$ for the associated Shimura variety. It is a quasi-projective variety defined over a number field $\mathbf{E}$ (the reflex field). The field $\mathbf{E}$ is the field of definition of the $\mathbf{G}(\CC)$-conjugacy class of the cocharacter $\mu\colon \GG_{\mathrm{m},\CC}\to \mathbf{G}_\CC$ naturally attached to the Shimura datum $(\mathbf{G},\mathbf{X})$. 

Let $p$ be a prime number, and suppose that $K$ can be written as $K=K_pK^p$ where $K_p\subset \mathbf{G}(\QQ_p)$ is hyperspecial and $K^p\subset \mathbf{G}(\AA_f^p)$ is compact open (we say that $p$ is a place of good reduction). By definition of a hyperspecial subgroup, we can write $K_p=\Gcal(\ZZ_p)$ for a reductive $\ZZ_p$-group $\Gcal$ such that $\Gcal\otimes_{\ZZ_p}\QQ_p \simeq \mathbf{G}_{\QQ_p}$. For each place $v|p$ in $\mathbf{E}$, Kisin (\cite{Kisin-Hodge-Type-Shimura} and Vasiu (\cite{Vasiu-Preabelian-integral-canonical-models}) have constructed a smooth, canonical model $\Scal_K$ of $\Sh_K(\mathbf{G},\mathbf{X})$ over the ring $\Ocal_{\mathbf{E}_v}$. We write $S_K\colonequals \Scal_K\otimes_{\Ocal_{\mathbf{E}_v}} k$ for the special fiber of $\Scal_K$, where $k$ is an algebraic closure of the residue field of $v$. 

Define $G\colonequals \Gcal\otimes_{\ZZ_p} \FF_p$. We may assume that it extends to a cocharacter of $\Gcal_{\overline{\ZZ}_p}$ (see \cite[\S 2.5]{Imai-Koskivirta-vector-bundles} for details). Write again $\mu\colon \GG_{\mathrm{m},k}\to G_k$ for its special fiber. In particular, the pair $(G,\mu)$ gives rise to a stack of $G$-zips of type $\mu$ as explained in \S\ref{zip-def}. Zhang (\cite{Zhang-EO-Hodge}) has constructed a smooth morphism of stacks
\begin{equation}\label{zeta-eq}
\zeta\colon S_K \to \GZip^\mu
\end{equation}
which is also surjective by \cite[Corollary 3.5.3(1)]{Shen-Yu-Zhang-EKOR}. The program started by W. Goldring and the author in \cite{Goldring-Koskivirta-global-sections-compositio, Goldring-Koskivirta-GS-cone,Goldring-Koskivirta-divisibility} is aimed at studying the geometry of $S_K$ by using $\GZip^\mu$ as a group-theoretical analogue. The authors proved in many cases that geometric information pertaining to $S_K$ can be read off the stack $\GZip^\mu$.

\subsection{Vector bundles on quotient stacks} \label{sec-vb-quot}

We recall the construction of vector bundles on quotient stacks attached to algebraic representations. Let $H$ be a smooth connected algebraic group over a field $k$, and $X$ a $k$-variety endowed with an algebraic action of $H$ (we say that $X$ is an $H$-variety). Let $\rho\colon H\to \GL(V)$ be an algebraic representation of $H$ on a finite-dimensional $k$-vector space $V$. Then, the "associated sheaf construction" of \cite[I.5.8]{jantzen-representations} produces a vector bundle $\Vcal(\rho)$ on the quotient stack $[H\backslash X]$. It is represented geometrically by the stack
\begin{equation}
    \left[ H\backslash (X\times V) \right]
\end{equation}
where $H$ acts diagonally on $X\times V$. The first projection $X\times H\to X$ induces a natural map to the quotient stack $[H\backslash X]$. Furthermore, the rank of $\Vcal(\rho)$ coincides with the dimension of $V$. In particular, characters $\lambda\colon H\to \GG_{\mathrm{m}}$ give rise to line bundles on $[H\backslash X]$. We will use the notation $\Lcal(\lambda)$ to denote the line bundle attached to $\lambda$.

We apply this general discussion to the stack of $G$-zips. Fix a cocharacter datum $(G,\mu)$ and write $\Zcal_\mu=(G,P,Q,L,M)$ for the attached zip datum, as defined in \S\ref{zip-def}. Let $\GZip^\mu =[E\backslash G_k]$ be the associated stack of $G$-zips. For any algebraic representation $\rho\colon E\to \GL(V)$ on a finite-dimensional $k$-vector space $V$, we obtain a vector bundle $\Vcal(\rho)$ on the stack $\GZip^\mu$. When $\rho\colon P\to \GL(V)$ is an algebraic $P$-representation, we view it as an $E$-representation via the first projection $\pr_1\colon E\to P$, and write again $\Vcal(\rho)$ for the attached vector bundle. To avoid ambiguity with the vector $\Vcal_I(\lambda)$ that will appear in section \ref{sec-zip-flag}, we write $\Lcal(\lambda)$ (instead of $\Vcal(\lambda)$) to denote the line bundle on $\GZip^\mu$ attached to a character $\lambda\in X^*(L)$.

\subsection{Picard group of a quotient stack}\label{sec-pic}
In this section, we recall general facts about Picard groups of quotient stacks explained in \cite[\S 2]{Koskivirta-Wedhorn-Hasse}. Let $k$ be an algebraically closed field and $H$ a smooth connected algebraic group over $k$. Let $X$ be a normal $H$-variety and write $a\colon H\times X\to X$ for the action map of $H$ on $X$. Define $p_2:H\times X\to X$ and $p_{23}\colon H\times H\times X\to H\times X$ as the projection maps given by $(g,x)\mapsto x$ and $(g_1,g_2,x)\mapsto (g_2,x)$ respectively. Let $\mu_H$ denote the multiplication map $H\times H\to H$. For a line bundle $\Lcal$ on $X$, an $H$-linearization of $\Lcal$ is an isomorphism $\phi : a^*(\mathscr{L})\to p_{2}^*(\mathscr{L})$ satisfying the cocycle condition
$$ p_{23}^*(\phi) \circ (id_H \times a)^*(\phi) =(\mu_H \times id_X)^*(\phi).$$
Let $\Pic^H(X)$ denote the group of isomorphism classes of $H$-linearized line bundles on $H$. It identifies naturally with the Picard group $\Pic([H\backslash X])$ of the quotient stack $[H\backslash X]$. Define $\Ecal(X) := \mathcal{O}(X)^{\times} / k^{\times}$, where $\mathcal{O}(X)^{\times}$ is the group of nowhere vanishing functions on $X$. When $X$ is an integral $k$-scheme of finite type, $\Ecal(X)$ is a finitely generated free abelian group (\cite[\S 1.3]{Knop-Kraft-Vust-G-variety}). Recall the following result: 

\begin{proposition}[{\cite[Corollary 2.3.1]{Koskivirta-Wedhorn-Hasse}}]\label{exseq}
Let $H$ be a smooth connected algebraic group and $X$ a normal variety endowed with an algebraic action of $H$. Then, there is an exact sequence of abelian groups
\[
1 \to k^{\times} \to (\mathcal{O}(X)^{\times})^H \to \Ecal(X) \to X^{*}(H) \to \Pic^{H}(X) \to \Pic(X) \to \Pic(H).
\]
\end{proposition}
Here, $(\mathcal{O}(X)^{\times})^H$ denotes the group of $H$-invariant elements of $\mathcal{O}(X)^{\times}$. Via the identification $\Pic^{H}(X)=\Pic([H\backslash X])$, the map $X^*(H)\to \Pic^H(X)$ is given by the construction $\lambda\mapsto \Lcal(\lambda)$ explained in section \ref{sec-vb-quot}. 

We apply the above discussion in the case of the action of $E$ on $G_k$ giving rise to the stack $\GZip^\mu=[E\backslash G_k]$. To simplify notation, we write $G$ instead of $G_k$. By \cite[\S 1.3]{Knop-Kraft-Vust-G-variety}, the natural map $X^*(G)\to \Ecal(G)$ is an isomorphism. Identifying $X^*(E)=X^*(L)$, we deduce that there is an exact sequence
\begin{equation}\label{Pic-zip}
1 \to X^*(G) \to X^{*}(L) \to \Pic(\GZip^\mu) \to \Pic(G).
\end{equation}
Since $\Pic(G)$ is finite by \cite[Proposition 4.5]{Knop-Kraft-Luna-Vust-Local-properties}, we obtain an isomorphism
\begin{equation}
    \left(X^*(L)/X^*(G)\right)\otimes_\ZZ \QQ \to \Pic(\GZip^\mu)_\QQ.
\end{equation}

\section{\texorpdfstring{Global sections on $\GZip^\mu$}{}}\label{glob-zip-ch2}

In this section, we first study in general the space $H^0(\GZip^\mu,\Vcal(\rho))$ of global sections of vector bundles over the stack $\GZip^\mu$. We then specialize to the case of line bundles.

\subsection{Brylinski–-Kostant filtration}

Before we recall the main theorem of \cite{Imai-Koskivirta-vector-bundles}, we set notation pertaining to representation theory of reductive groups. Fix an algebraic $B$-representation $\rho\colon B\to \GL(V)$ and let $V=\bigoplus_{\nu \in X^*(T)}V_\nu$ be the weight decomposition of $V$. For any $v\in V_\nu$ and $\alpha\in \Phi$, we can write uniquely
\[
 u_{\alpha}(x)v=\sum_{j \geq 0} x^j E_{\alpha}^{(j)}(v), \quad \forall x\in \GG_{\mathrm{a}},
\]
for elements $E_{\alpha}^{(j)}(v) \in V_{\nu+j\alpha}$ (\cite[Lemma 3.3.1]{Imai-Koskivirta-vector-bundles}). Extend the definition of $E_{\alpha}^{(j)}(v)$ by additivity to any $v\in V$. This defines a map $E_{\alpha}^{(j)} \colon V \to V$ for any $j\geq 0$ and $\alpha\in \Phi$. When $j<0$, put $E_{\alpha}^{(j)}=0$.

Write $\sigma\in \Gal(k/\FF_p)$ for the $p$-power Frobenius element. Let $\wp \colon T \to T$ be the Lang map defined by $g \mapsto g \varphi(g)^{-1}$. It induces an isomorphism 
\begin{equation}\label{def-Pstar}
 \wp_* \colon X_*(T)_{\RR} \stackrel{\sim}{\longrightarrow} X_*(T)_{\RR}; \  \delta \mapsto \wp \circ \delta =\delta - p \sigma(\delta).     
\end{equation}
For $\alpha\in \Delta$, set $\delta_{\alpha}=\wp_*^{-1}(\alpha^{\vee})$. For $\alpha\in \Delta^P$, define an integer $m_\alpha$ by 
\begin{equation}\label{malpha-equ}
 m_{\alpha} =\min \{ m \geq 1 \mid 
 \sigma^{-m}(\alpha) \notin I \}. 
\end{equation}
For example, if $P$ is defined over $\FF_p$, then $m_\alpha=1$ for all $\alpha\in \Delta^P$.

Let $m\geq 1$ be an integer and let $\Xi=(\alpha_1, \ldots, \alpha_m) \in \Phi^m$ be an $m$-tuple of roots. Suppose $H$ is a closed subgroup scheme of $G$ contaning 
$T$ and $U_{\alpha_i}$ for all $1 \leq i \leq m$. 
Let $V$ be a finite dimensional algebraic representation of $H$. For $\mathbf{a}=(a_1,\ldots,a_m) \in (k^{\times})^m$ and 
$\mathbf{r}=(r_1,\ldots,r_m) \in \RR^m$, define
\begin{align*}
 (\ZZ^m)_{\mathbf{r}} &= \left\{ (n_1,\ldots,n_m) \in \ZZ^m \relmiddle| \sum_{i=1}^m n_i r_i =0 
 \right\}, \\ 
 \Lambda_{\Xi,\mathbf{r}} &= \left\{ \sum_{i=1}^m n_i \alpha_i \relmiddle| 
 (n_1,\ldots,n_m) \in (\ZZ^m)_{\mathbf{r}}
 \right\}. 
\end{align*}
For a character $\nu\in X^*(T)$, write $[\nu]$ for its coset in the quotient $X^*(T)/\Lambda_{\Xi,\mathbf{r}}$. Furthermore, define
\[
 V_{[\nu]}=\bigoplus_{\nu \in [\nu]} V_{\nu} . 
\] 
Let $\mathbf{j}=(j_1,\ldots,j_m) \in \ZZ^m$ and write again $[\mathbf{j}] \in \ZZ^m/(\ZZ^m)_{\mathbf{r}}$ for its residue class. Define 
\begin{align*}
 [\mathbf{j}] \cdot \mathbf{r} 
 &= \sum_{i=1}^m j_i r_i \in \RR, \\ 
 [\nu] +[\mathbf{j}] \cdot \Xi 
 &= \left[ \nu +\sum_{i=1}^m j_i \alpha_i \right] \in 
 X^*(T)/\Lambda_{\Xi,\mathbf{r}}
\end{align*}
(note that these elements are well-defined). Let $[\nu] \in X^*(T)/\Lambda_{\Xi,\mathbf{r}}$ and let $\delta \colon X^*(T) \to \RR$ be any function. We define a subspace
$\fil_{\delta}^{\Xi,\mathbf{a},\mathbf{r}} V_{[\nu]}\subset V_{[\nu]}$ as follows:

\begin{equation*}
 \fil_{\delta}^{\Xi,\mathbf{a},\mathbf{r}} V_{[\nu]}\colonequals \bigcap_{[\mathbf{j}] \in \ZZ^m/(\ZZ^m)_{\mathbf{r}}} 
 \bigcap_{\substack{\chi \in [\nu] +[\mathbf{j}] \cdot \Xi, \\ [\mathbf{j}] \cdot \mathbf{r} > \delta(\chi)}}
 \Ker \left( \sum_{\mathbf{j} \in [\mathbf{j}]} 
 \pr_{\chi} \circ  a_1^{j_1} E_{\alpha_1}^{(j_1)} \circ \cdots \circ a_m^{j_m} E_{\alpha_m}^{(j_m)} \colon V_{[\nu]
 } \to 
 V_{\chi} 
 \right) 
\end{equation*}
where $\pr_{\chi} \colon V_{[\nu] +[\mathbf{j}] \cdot \Xi} \to V_{\chi}$ is the natural projection. This filtration is similar to the Brylinski–-Kostant filtration of a representation.

\subsection{The space of global sections}\label{glob-subsec}

We explain the results of \cite{Imai-Koskivirta-vector-bundles}, where N. Imai and the author determined explicitly the space of global sections $H^0(\GZip^\mu,\Vcal(\rho))$ for a general algebraic $P$-representation $\rho$. We first recall the following easy result (\cite[Lemma 1.2.1]{Koskivirta-automforms-GZip}):
\begin{proposition}\label{prop-H0-dim}
For any algebraic representation $\rho\colon P\to \GL(V)$, we have
\begin{equation}
 \dim_k H^0(\GZip^\mu,\Vcal(\rho)) \leq \dim_k(V)
\end{equation}
In particular, for any $\lambda\in X^*(L)$, the $k$-vector space $H^0(\GZip^\mu,\Lcal(\lambda))$ has dimension $\leq 1$.
\end{proposition}
For a general $P$-representation $P\to \GL(V)$, the space $H^0(\GZip^\mu,\Vcal(\rho))$  of global sections can be expressed as the part of the Brylinski–-Kostant filtration of $V$ invariant under the action of a certain (non-smooth) finite algebraic subgroup $L_\varphi\subset L$ (\cite[Theorem 3.4.1]{Imai-Koskivirta-vector-bundles}). Specifically, define first $L_{\varphi}\subset E$ as the scheme-theoretical stabilizer of the element $1$. Explicitly,
\begin{equation}\label{Lvarphi-eq}
    L_\varphi = E\cap \{(x,x) \mid x\in G_k \}.
\end{equation}
This algebraic group is in general non-smooth. Via the first projection $\pr_1\colon E\to P$, identify $L_\varphi$ with an algebraic subgroup of $P$. One can then show (\cite[Lemma 3.2.1]{Koskivirta-Wedhorn-Hasse}) that $L_{\varphi}$ is a finite group-scheme contained in $L$. Furthermore, let $L_0\subset L$ be the largest algebraic subgroup contained in $L$, i.e. $L_0\colonequals \bigcap_{r\in \ZZ} \sigma^r(L)$. Then $L_{\varphi}$ can be written as a semidirect product
\begin{equation}
L_{\varphi}=L_{\varphi}^\circ\rtimes L_0(\FF_p)    
\end{equation}
where $L_{\varphi}^\circ$ is the identity component of $L_{\varphi}$, which is a finite unipotent group-scheme. When $P$ is defined over $\FF_p$, we simply have $L_0=L$ and $L_{\varphi}=L(\FF_p)$, viewed as an etale group-scheme.

For $\alpha \in \Delta^P$, let $m_\alpha$ be the integer defined by \eqref{malpha-equ}, put $\mathbf{a}_\alpha=(-1,\ldots, -1) \in (k^{\times})^{m_\alpha}$ and define
\[\Xi_{\alpha}=(-\alpha, \sigma^{-1}(\alpha),\ldots,
 \sigma^{-(m_{\alpha}-1)}(\alpha)).\]
Define also $\mathbf{r}_{\alpha}=(r_{\alpha,1}, \ldots, r_{\alpha,m_{\alpha}})$, 
where $r_{\alpha,1}=1-\langle \alpha,\delta_{\alpha}  \rangle$ and 
\[
 r_{\alpha,i}=\frac{\langle \alpha,\delta_{\alpha}  \rangle -1}{p^{i-1}}  
\]
for $2 \leq i \leq m_{\alpha}$. 
We view $\delta_{\alpha}$ as 
a function $\delta_\alpha\colon X^*(T) \to \RR$ 
given by $\delta_\alpha(\chi)= \langle \chi ,\delta_{\alpha} \rangle$.

\begin{theorem}\label{thm-main-H0}
Let $\rho\colon P\to \GL(V)$ be an algebraic $P$-representation. One has an identification
\begin{equation}\label{eq-main}
H^0(\GZip^\mu,\mathcal{V}(\rho))=V^{L_{\varphi}}\cap 
 \bigcap_{\alpha \in \Delta^P} 
 \bigoplus_{[\nu] \in X^*(T)/\Lambda_{\Xi_{\alpha},\mathbf{r}_{\alpha}}} 
 \fil_{\delta_{\alpha}}^{\Xi_{\alpha},\mathbf{a}_{\alpha},\mathbf{r}_{\alpha}} V_{[\nu]} .
\end{equation}
\end{theorem}

\subsection{Global sections of line bundles}

In this section, we specialize to the case when the representation $\rho\colon P\to \GL(V)$ is the one-dimensional representation attached to a character $\lambda\in X^*(L)=X^*(P)$. In this case, the statement of Theorem \ref{thm-main-H0} simplifies considerably. Recall that $H^0(\GZip^\mu,\Lcal(\lambda))$ has dimension $\leq 1$. We want to determine for which characters $\lambda\in X^*(L)$ this space is nonzero.

\begin{proposition}\label{prop-line}
For a character $\lambda\in X^*(L)$, the following are equivalent:
\begin{equivlist}
    \item The space $H^0(\GZip^\mu,\Lcal(\lambda))$ is nonzero.
    \item $\lambda$ is trivial on $L_\varphi$ and $\langle \lambda, \delta_\alpha \rangle \geq 0$ for all $\alpha\in \Delta^P$.
\end{equivlist}
Furthermore, in this case $H^0(\GZip^\mu,\Lcal(\lambda))$ is one-dimensional.    
\end{proposition}

\begin{proof}
    Write $V_I(\lambda)$ for the underlying one-dimensional vector space of $\lambda$. In this case, for dimension reasons we find that 
\begin{equation}
\fil_{\delta_{\alpha}}^{\Xi_{\alpha},\mathbf{a}_{\alpha},\mathbf{r}_{\alpha}} V_{\lambda}=V_\lambda=V_I(\lambda)    
\end{equation}
if $\langle \lambda, \delta_\alpha \rangle \geq 0$ and $\fil_{\delta_{\alpha}}^{\Xi_{\alpha},\mathbf{a}_{\alpha},\mathbf{r}_{\alpha}} V_{\lambda}=0$ otherwise. On the other hand, the space $V_I(\lambda)^{L_\varphi}$ is zero if and only if $\lambda$ is trivial on the finite subgroup $L_\varphi\subset L$. The equivalence of (i) and (ii) then follows easily from Theorem \ref{thm-main-H0}. When this condition is satisfied, $H^0(\GZip^\mu,\Lcal(\lambda))$ is one-dimensional by Proposition \ref{prop-H0-dim}.
\end{proof}

\begin{corollary}\label{cor-line-H0}
Let $\lambda\in X^*(L)$ and assume that $\langle \lambda, \delta_\alpha \rangle \geq 0$ for all $\alpha\in \Delta^P$. Then there is an integer $N\geq 1$ such that $H^0(\GZip^\mu,\Lcal(\lambda)^{\otimes N}) \neq 0$.
\end{corollary}

\begin{proof}
By considering an appropriate multiple $N\lambda$ ($N\geq 1$), we can arrange that $N\lambda$ vanishes on the finite group $L_\varphi$. Since $\Lcal(N\lambda)=\Lcal(\lambda)^{\otimes N}$, the result follows from Proposition \ref{prop-line}.
\end{proof}

The set of $\lambda$ such that some positive power of $\Lcal(\lambda)$ admits nonzero global sections is thus defined inside $X^*(L)$ as the intersection of the half-planes given by $\langle \lambda, \delta_\alpha \rangle \geq 0$. We note the following:

\begin{proposition}\label{prop-lin-indep}
The linear forms $\{\lambda\mapsto \langle \lambda,\delta_\alpha\rangle\}_{\alpha\in \Delta^P}$ are linearly independent, when viewed as linear forms on $X^*(L)_\QQ$. Equivalently, they form a basis of the dual vector space of $\left(X^*(L)/X^*(G)\right)\otimes_{\ZZ} \QQ$
\end{proposition}

\begin{proof}
Let $D\subset X^*(L)$ be the set of $\lambda\in X^*(L)$ such that $\langle \lambda,\delta_\alpha\rangle=0$ for all $\alpha\in \Delta^P$. Let $\lambda\in D$ and $\alpha\in \Delta^P$. By definition of $\delta_\alpha$, we have 
\begin{equation}
    \delta_\alpha-p\sigma(\delta_\alpha)=\delta_\alpha-p\delta_{\sigma(\alpha)}=\alpha^\vee.
\end{equation}
If $\sigma(\alpha)\in \Delta^P$, then we obtain $\langle \lambda,\alpha^\vee\rangle=0$. Otherwise, we have $\sigma(\alpha)\in I$. Let $r\geq 1$ be the smallest integer such that $\sigma^r(\alpha)\in \Delta^P$. Then, since $\lambda\in X^*(L)$ and $\sigma(\alpha),\dots,\sigma^{r-1}(\alpha)$ are in $I$, we obtain inductively
\begin{equation*}
    \langle \lambda,\alpha^\vee\rangle= -p\langle \lambda,\delta_{\sigma(\alpha)}\rangle \\
= -p^r\langle \lambda,\delta_{\sigma^r(\alpha)}\rangle =0.
\end{equation*}
Therefore, $\lambda$ is orthogonal to all $\alpha^\vee$ for $\alpha\in \Delta$, which implies that $\lambda\in X^*(G)$. Since we have $\rank(X^*(L))-\rank(X^*(G)) = |\Delta^P|$, dimension considerations show that the linear forms $\lambda\mapsto \langle \lambda,\delta_\alpha\rangle$ on $X^*(L)_\QQ$ are linearly independent.
\end{proof}

We finish this section with a useful lemma. Let $H$ be a smooth connected algebraic $k$-group acting on a $k$-variety $X$. For $\lambda\in X^*(H)$, denote by $\Lcal(\lambda)$ the associated line bundle on the quotient stack $[H\backslash X]$ (section \ref{sec-vb-quot}). By definition, the space of global sections of $\Lcal(\lambda)$ can be identified with the space of regular maps $f\colon X\to \AA^1$ satisfying
\begin{equation}\label{H0-quot-eq}
    f(h\cdot x)=\lambda(h)f(x),\quad h\in H, x\in X.
\end{equation}
We will use the following lemma:
\begin{lemma}\label{lemma-B-OK}
Let $B_H\subset H$ be a Borel subgroup (in the sense that $H/B_H$ is a projective variety) and $\lambda\in X^*(H)$ a character. Let $f\colon X\to \AA^1$ be a regular map satisfying $ f(h\cdot x)=\lambda(h)f(x)$ for all $h\in B_H$ and $x\in X$. Then $f$ satisfies the same relation for all $h\in H$ and $x\in X$. In other words, $f$ identifies with a global section of $\Lcal(\lambda)$ on $[H\backslash X]$.
\end{lemma}

\begin{proof}
Let $x\in X$ be fixed, and consider the map $\psi\colon H\to \AA^1$ defined by
\begin{equation}
    \psi(h)\colonequals \frac{f(h\cdot x)}{\lambda(h)}.
\end{equation}
It is clear that $\psi$ is a regular map. By assumption, for any $b\in B_H$, $h\in H$, we have $\psi(bh)=\psi(h)$. Since the quotient $B_H\backslash H$ is proper, we deduce that $\psi$ is constant, hence equal to $f(g)$. The result follows.
\end{proof}

By definition, a section $f\in H^0(\GZip^\mu,\Lcal(\lambda))$ (for $\lambda\in X^*(L)$) can be interpreted as a regular map $f\colon G_k\to \AA^1$ satisfying the relation
\begin{equation}\label{Gzip-glob-eq}
    f(xgy^{-1}) = \lambda(x)f(g)
\end{equation}
for all $(x,y)\in E$ and $g\in G_k$. Now, fix a Borel pair $(B,T)$ as explained in section \ref{sec-gp-not}. Define a subgroup $E'\subset E$ by 
\begin{equation}\label{Eprime-eq}
E'\colonequals E\cap (B\times G).    
\end{equation}
In other words, $E'$ is the subgroup of pairs $(x,y)\in E$ such that $x\in B$. One sees easily that $E/E'\simeq P/B$ (hence is projective). From Lemma \ref{lemma-B-OK}, we deduce the following: 

\begin{corollary}\label{cor-Eprime-E}
Let $\lambda\in X^*(L)$ and let $f\colon G_k\to \AA^1$ be a regular map satisfying \eqref{Gzip-glob-eq} for all $(x,y)\in E'$ and $g\in G_k$. Then the same relation holds for all $(x,y)\in E$ and $g\in G_k$. Hence $f$ identifies with an element of $H^0(\GZip^\mu,\Lcal(\lambda))$.
\end{corollary}

\subsection{\texorpdfstring{$\mu$-ordinary Hasse invariants}{}}

We recall the main theorem of \cite{Koskivirta-Wedhorn-Hasse} by the author and T. Wedhorn. In \loccitn, the following result was proved:

\begin{proposition}[{\cite[Theorem 5.1.4]{Koskivirta-Wedhorn-Hasse}}] \label{prop-KosWed}
Assume that $\lambda\in X^*(L)$ satisfies
\begin{equation}\label{char-L-reg}
\langle \lambda,\alpha^\vee \rangle<0, \quad \forall \alpha\in \Delta^P.
\end{equation}
Then the space $H^0(\GZip^\mu,\Lcal(\lambda)^{\otimes N_\mu})$ is nonzero for a certain integer $N_\mu\geq 1$. Furthermore, any nonzero section $h\in H^0(\GZip^\mu,\Lcal(\lambda)^{\otimes N_\mu})$ is a $\mu$-ordinary Hasse invariant, in the sense that the non-vanishing locus of $h$ is precisely the complement of the $\mu$-ordinary locus $\Ucal_\mu\subset \GZip^\mu$.
\end{proposition}

By pullback via the map $\zeta\colon S_K\to \GZip^\mu$, the authors also deduced the existence of such sections on the Shimura variety $S_K$. We explain below the connection with Corollary \ref{cor-line-H0}. For each $\alpha\in \Delta^P$, fix an integer $d_\alpha\geq 1$ such that $\alpha$ is defined over $\FF_{p^{d_\alpha}}$. We can write
\begin{equation}
    \delta_\alpha=-\frac{1}{p^{d_\alpha}-1}\sum_{i=0}^{d_\alpha-1} p^i \sigma^i(\alpha^\vee).
\end{equation}
The condition $\langle \lambda, \delta_\alpha \rangle \geq 0$ can thus be written as
\begin{equation}\label{eq-sum}
    \sum_{i=0}^{d_\alpha-1} p^i \langle \lambda,\sigma^i(\alpha^\vee)\rangle \leq 0.
\end{equation}
We claim that if $\lambda \in X^*(L)$ satisfies \eqref{char-L-reg}, then it is anti-dominant, i.e. $\langle \lambda,\alpha^\vee \rangle \leq 0$ for all $\alpha\in \Phi^+$. Indeed, we may write $\alpha$ as linear combination with positive coefficients of simple roots. Since $\langle \lambda,\alpha^\vee\rangle=0$ for $\alpha\in I$, we deduce the claim. In particular, when $\lambda \in X^*(L)$ satisfies \eqref{char-L-reg}, all terms in the sum \eqref{eq-sum} are non-positive, and the one corresponding to $i=0$ is negative. Thus we see that Proposition \ref{prop-KosWed} is compatible with Corollary \ref{cor-line-H0}.

We now generalize the construction of $\mu$-ordinary Hasse invariants from \cite{Koskivirta-Wedhorn-Hasse}. Unfortunately, we cannot use the statement of Theorem \ref{thm-main-H0} as is. We need to return to its proof, following \cite[Theorem 3.4.1]{Imai-Koskivirta-vector-bundles}, in order to state a more precise result. First, we recall the following: For any algebraic $P$-representation $P\to \GL(V)$, the space of global sections of $\Vcal(\rho)$ over the $\mu$-ordinary locus $\Ucal_\mu$ can be identified as
\begin{equation}\label{H0-mu-eq}
    H^0(\Ucal_\mu,\Vcal(\rho)) = V^{L_\varphi}.
\end{equation}
Furthermore, via the identification given by Theorem \ref{thm-main-H0}, the natural injective map $H^0(\GZip^\mu,\Vcal(\rho)) \to H^0(\Ucal_\mu,\Vcal(\rho))$ corresponds to the natural inclusion into $V^{L_\varphi}$.

To simplify, we consider the case of a line bundle $\Lcal(\lambda)$ (for $\lambda\in X^*(L)$). If we write $N_\mu$ for the order of the finite group-scheme $L_\varphi$, then for any $\lambda\in X^*(L)$, the character $N_\mu \lambda$ vanishes on $L_\varphi$. We deduce that for any $\lambda\in X^*(L)$, the space $H^0(\Ucal_\mu,\Lcal(\lambda)^{\otimes N_\mu})$ is one-dimensional. We may view an element $h\in H^0(\Ucal_\mu,\Lcal(\lambda)^{\otimes N_\mu})$ as a regular map $h\colon U_\mu\to \AA^1$ satisfying
\begin{equation*}
    h(xgy^{-1}) = \lambda(x)^{N_\mu}f(g), \quad (x,y)\in E, g\in U_\mu.
\end{equation*}
Since $G_k$ is a normal $k$-variety, we may speak of the divisor $\div(h)$ of $h$. By $E$-equivariance of $h$, the divisor $\div(h)$ can be written as
\begin{equation*}
    \div(h)=\sum_{\alpha\in \Delta^P} \mult_\alpha(h) \overline{Z}_{\alpha}
\end{equation*}
for certain multiplicities $\mult_{\alpha}(h)$, where $\{Z_\alpha\}_{\alpha\in \Delta^P}$ are the codimension one $E$-orbits in $G_k$ (section \ref{subsec-zipstrata}). Define the sign function $\sgn\colon \RR\to \{-1,1,0\}$ by
\begin{equation}
\sgn(x)=\begin{cases}
    1 & \textrm{if} \ x>0 \\
    -1 & \textrm{if} \ x<0 \\
        0 & \textrm{if} \ x=0.
\end{cases}
\end{equation}
The proof of \cite[Theorem 3.4.1]{Imai-Koskivirta-vector-bundles} shows the following: 
\begin{lemma}\label{sgn-prop} One has $\textup{\sgn}(\mult_{\alpha}(h)) = \textup{\sgn}(\langle \lambda,\delta_\alpha \rangle)$ for all $\alpha\in \Delta^P$.
\end{lemma}

To prove Lemma \ref{sgn-prop}, we review a construction explained in \cite[\S 3.1]{Imai-Koskivirta-vector-bundles}. For $\alpha \in \Phi$, denote by $\phi_\alpha  \colon  \SL_{2,k} \to G_k$ the unique homomorphism such that 
\[
 \phi_\alpha 
 \left( \begin{pmatrix}
 1 & x \\ 0 & 1 
 \end{pmatrix}\right) = u_{\alpha}(x), \quad 
 \phi_\alpha 
 \left( \begin{pmatrix}
 1 & 0 \\ x & 1 
 \end{pmatrix}\right) = u_{-\alpha}(x) 
\]
afforded by \cite[9.2.2]{Springer-Linear-Algebraic-Groups-book}. It also satisfies $\phi_{\alpha}(\diag(t,t^{-1}))=\alpha^\vee(t)$. For $\alpha\in \Delta^P$, recall that $Z_\alpha=E\cdot s_\alpha$ is the $E$-orbit of $s_\alpha$. For any $\alpha\in \Delta^P$, define a subset $X_\alpha$ as follows:
\begin{equation}
X_\alpha  \colonequals  G_k \setminus \bigcup_{\beta\in \Delta^P,\ \beta\neq \alpha}Z_\beta.
\end{equation}
It is clear that $X_\alpha$ is open and contains the open $E$-orbit $U_\mu$. Set $Y=E \times \AA^1$ and define the map
\begin{equation}\label{phia}
\psi_\alpha \colon Y \to G_k ; \ ((x,y),t)\mapsto x \phi_{\alpha} \left(A(t)\right)y^{-1} \quad \textrm{where} \ A(t)=\left(
\begin{matrix}
t & 1 \\ -1 & 0 \end{matrix}
\right)\in \SL_{2,k}.
\end{equation}
One has $\phi_{\alpha} (A(0))=s_\alpha$ in $W$. The properties of the map $\psi_\alpha$ are given by the following proposition

\begin{proposition}[{\cite[Proposition 3.1.4]{Imai-Koskivirta-vector-bundles}}]\label{psiadapted} \ 
\begin{assertionlist}
\item \label{item-imagepsi} The image of $\psi_\alpha$ is contained in $X_\alpha$.
\item \label{item-psit} For any $(x,y)\in E$ and $t\in \AA^1$, one has $\psi_\alpha((x,y),t)\in U_{\mu} \Longleftrightarrow t\neq 0$.
\item \label{item-zero} For all $(x,y)\in E$, we have $\psi_\alpha((x,y),0)\in E\cdot s_\alpha$.
\end{assertionlist}
\end{proposition}
Since $h\colon U_\mu \to \AA^1$ is an $E$-equivariant function, we can apply the theory of "adapted morphisms" established in \cite[\S 3.2]{Koskivirta-automforms-GZip}. We deduce that the sign of $\mult_\alpha(h)$ coincides with the sign of the multiplicity of the divisor of the composition $h\circ \psi_\alpha$ (\loccitn, Corollary 3.2.3). Note that $A(t)$ can be decomposed as follows:
\begin{equation}\label{At}
A(t)=\left(\begin{matrix}
1&0\\-t^{-1}&1
\end{matrix} \right) \left(\begin{matrix}
t&0\\0&t^{-1}
\end{matrix} \right)\left(\begin{matrix}
1&t^{-1}\\0&1
\end{matrix} \right).
\end{equation}
Since $\alpha\in \Delta^P$, the image by $\phi_\alpha$ of the matrices $\left(\begin{matrix}
1&0\\-t^{-1}&1
\end{matrix} \right)$ and $\left(\begin{matrix}
1&t^{-1}\\0&1
\end{matrix} \right)$ lie in $R_{\mathrm{u}}(P)$ and $R_{\mathrm{u}}(Q)$ respectively. We deduce
\begin{equation}
   h(\psi_\alpha((x,y),t)) = \lambda(x) h \left( \phi_{\alpha}\left( \left(\begin{matrix}
t&0\\0&t^{-1}
\end{matrix} \right) \right) \right) =  \lambda(x) h(\alpha^\vee(t)). 
\end{equation}
Since $\alpha^\vee =\delta_\alpha-p\sigma(\delta_\alpha)$, we obtain $h(\alpha^\vee(t))=\lambda(\delta_\alpha(t))=t^{\langle \lambda,\delta_\alpha \rangle}$. Hence, the sign of the divisor of $h\circ \psi_\alpha$ coincides with $\langle \lambda,\delta_\alpha \rangle$. This finishes the proof of Lemma \ref{sgn-prop}. We deduce from this discussion the following theorem, which generalizes the existence of $\mu$-ordinary Hasse invariants of \cite{Koskivirta-Wedhorn-Hasse} and completes Proposition \ref{prop-line} above.
\begin{theorem}\label{thm-mu}
For $\lambda\in X^*(L)$, the following are equivalent.
\begin{equivlist}
    \item There exists a $\mu$-ordinary Hasse invariant $h\in H^0(\GZip^\mu,\Lcal(\lambda)^{\otimes N_\lambda})$.
    \item $\langle \lambda, \delta_\alpha \rangle > 0$ for all $\alpha\in \Delta^P$.
\end{equivlist}
\end{theorem}
When $P$ is defined over $\FF_p$, one has $\delta_\alpha = - \frac{\alpha^\vee}{p-1}$, hence the content of Theorem \ref{thm-mu} coincides with Proposition \ref{prop-KosWed}. However, for more general situations, for example the case of a Weil restriction, the condition (ii) of Theorem \ref{thm-mu} is less strenuous than that of Proposition \ref{prop-KosWed}.

\section{\texorpdfstring{The Cox ring of $\GZip^\mu$}{}}

\subsection{Mori dream spaces}

We briefly review the theory of Mori dream spaces. The natural setting to define these objects is that of projective varieties. In the context of this paper, we will consider the stacks of $G$-zips and $G$-zip flags (the latter is defined in section \ref{sec-zip-flag}). Although the classical theory of Mori dream spaces does not apply in this generality, we will see that all objects can be naturally defined for the stacks that we consider. Mori dream spaces were first defined by Hu and Keel in \cite{Hu-Keel-Mori-DS}. To simplify, we assume below that $X$ a smooth projective variety over an algebraically closed field $k$. Then $X$ is said to be a Mori dream space if it satisfies the following two conditions:

\begin{definitionlist}
    \item The Picard group $\Pic(X)$ is a finitely generated abelian group. 
    \item There are finitely many birational maps $f_i \colon X\to X_i$ (with $X_i$ projective) defined on some open subset of $X$, which are isomorphisms in codimension one, and satisfy that for any moveable divisor $D$ on $X$, there exists $i$ and a semiample divisor $D_i$ on $X_i$ such that $D=f_i^*(D_i)$.
    \end{definitionlist}
Let $X$ be a projective, smooth $k$-variety, such that $\Pic(X)$ is finitely generated. Fix a system of Weyl divisors $\underline{D}=(D_1,\dots, D_m)$ such that the attached line bundles $(\Ocal(D_1),\dots, \Ocal(D_m))$ form a basis of $\Pic(X)_\QQ$. For $\lambda=(\lambda_1,\dots, \lambda_m)\in \ZZ^m$ set $\lambda \cdot \underline{D}\colonequals \sum_{i=1}^m \lambda_i D_i$. The Cox ring of $X$ is defined as follows:

\begin{equation}
    \Cox(X)\colonequals \bigoplus_{\lambda\in \ZZ^m} H^0(X,\Ocal(\lambda \cdot \underline{D})).
\end{equation}
This ring depends on the choice of $\underline{D}$, but most of its properties are independent of this choice. The cone of effective divisor classes is the grading set of this graded algebra, namely it is given by:
\begin{equation}\label{effX}
    \Eff(X)=\left\{ \lambda\in \ZZ^m \ | \ H^0(X,\Ocal(\lambda \cdot \underline{D}))\neq 0 \right\}.
\end{equation}

\begin{theorem}[{\cite[Proposition 2.9]{Hu-Keel-Mori-DS}}]\label{thm-MDS-cox}
Let $X$ be a projective, smooth $k$-variety, such that $\Pic(X)$ is finitely generated. The following are equivalent:
\begin{equivlist}
    \item $X$ is a Mori dream space.
    \item The Cox ring $\Cox(X)$ is finitely generated over $k$.
\end{equivlist}
\end{theorem}

For example, smooth Fano varieties and spherical varieties are examples of Mori dream spaces. In particular, if $G$ is a connected reductive group over $k$ and $R\subset G$ is a parabolic subgroup, then $G/R$ is a Mori dream space. Let $(B,T)$ be a Borel pair such that $B\subset P$, as in section \ref{sec-gp-not}. We will consider the Cox ring of $P/B$ and $G/B$. Recall that $I\subset \Delta$ denotes the type of $P$. Consider the induced $P$-representation $V_I(\lambda)=\ind_B^P(\lambda)$ and define:
\begin{equation}\label{coxI-def-eq}
    \Cox_I(G)\colonequals \bigoplus_{\lambda\in X^*(T)} V_I(\lambda).
\end{equation}
Since we can also write $V_I(\lambda)=H^0(P/B,\Lcal(\lambda))$, the ring $\Cox_I(G)$ can be interpreted as the Cox ring of $P/B$. In particular, $\Cox_I(G)$ is a finitely generated $k$-algebra. Taking $P=G$, we also have the ring $\Cox_{\Delta}(G)$, which is the Cox ring of $G/B$.

\subsection{Graded rings and monoids}

Before we discuss the Cox ring for the stack of $G$-zips, we make some general observations on graded rings. Let $\Gamma$ be a finitely generated free abelian group, and let $X\subset \Gamma$ be a submonoid. We say that $X$ is finitely generated if there exists $x_1,\dots, x_k\in X$ such that $X=\ZZ_{\geq 0} x_1 + \dots + \ZZ_{\geq 0} x_k$. Write $X_{\QQ_{\geq 0}}$ for the set of linear combinations with non-negative rational coefficients of elements of $X$ in $\Gamma_\QQ$. We call $X_{\QQ_{\geq 0}}$ a $\QQ_{\geq 0}$-monoid. Say that $X_{\QQ_{\geq 0}}$ is finitely generated if there exists $x_1,\dots, x_k\in X_{\QQ_{\geq 0}}$ such that $X_{\QQ_{\geq 0}}=\QQ_{\geq 0} x_1 + \dots + \QQ_{\geq 0} x_k$ (in this case, one can assume that $x_1,\dots, x_k\in X$). Next, suppose
\begin{equation}
R=\bigoplus_{\lambda \in \Gamma} R_{\lambda}    
\end{equation}
is a graded $k$-algebra, where $R_{\lambda}$ is a $k$-vector space. Let $\Eff(R)$ denote the support of $R$, namely the set
\begin{equation}
    \Eff(R)\colonequals \{\lambda\in \Gamma \ | \ R_\lambda\neq 0\}.
\end{equation}
The notation is inspired from the geometric setting, where $\Eff(R)$ is the set of effective classes of divisors, as in \ref{effX}. We make the following assumptions:
\begin{bulletlist}
    \item $R$ is an integral domain.
\item $\Frac(R)$ is finitely generated over $k$.
\item $R$ is integrally closed
\item For all $\lambda\in \Gamma$, the $k$-vector space $R_\lambda$ has dimension $\leq 1$. 
\end{bulletlist}
We sometimes call the support $\Eff(R)$ the grading monoid of $R$. Indeed, when $R$ is an integral domain, one sees immediately that $\Eff(R)$ is a submonoid of $\Gamma$.

\begin{lemma}\label{lem-equiv-fg}
Under the above assumptions, the following assertions are equivalent:
\begin{equivlist}
\item $R$ is a finitely generated $k$-algebra.
\item $\Eff(R)$ is a finitely generated monoid.
\item $\Eff(R)_{\QQ_{\geq 0}}$ is a finitely generated $\QQ_{\geq 0}$-monoid. 
\end{equivlist}
\end{lemma}

\begin{proof}
Assume first that $R$ is finitely generated. Then we can write $R=k[f_1,\dots, f_n]$ for some homogeneous elements $f_i\in R_{\lambda_i}$ ($i=1,\dots ,n$). Since $R$ is an integral domain, one sees immediately that $\Eff(R)=\ZZ_{\geq 0}\lambda_1+\dots +\ZZ_{\geq 0}\lambda_N$, which proves (ii). The implication (ii) $\Rightarrow$ (iii) is obvious. Finally, assume that $\Eff(R)_{\QQ_{\geq 0}}=\QQ_{\geq 0} \lambda_1 + \dots + \QQ_{\geq 0} \lambda_N$ for some elements $\lambda_1,\dots,\lambda_N\in \Gamma$. After changing $\lambda_i$ to a positive multiple, we may assume $\lambda_i\in \Eff(R)$. By definition, there exists a nonzero element $f_i\in R_{\lambda_i}$. Denote by $X_0\subset \Eff(R)$ the monoid generated by $\lambda_1,\dots,\lambda_N$ and consider the subalgebra
\begin{equation}
    R'\colonequals k[f_1,\dots, f_N].
\end{equation}
If $f\in R_\lambda$ is nonzero for some $\lambda\in \Gamma$, then by definition $\lambda\in \Eff(R)$ and hence by assumption there exist integers $d\geq 1$ and $d_1, \dots, d_N\geq 0$ such that $d\lambda = d_1\lambda_1+\dots + d_N\lambda_N$. Set $f_0=\prod_{i=1}^N f_i^{d_i}$. Clearly $f_0$ lies in $R_{d\lambda}$. Since $R_{d\lambda}$ is one-dimensional by assumption, we have $f^d=a f_0$ for some $a\in k^*$. From this, we deduce that the ring extension $R'\subset R$ is integral. Write $K' \colonequals \Frac(R')$ and $K\colonequals \Frac(R)$ for the fields of fractions. Since $R$ is integral over $R'$, it is clear that $K$ is an algebraic extension of $K'$. Furthermore, since $K$ is finitely generated over $k$, we deduce that $K$ is a finite extension of $K'$. Recall the following result:

\begin{theorem}[{\cite[\S 3, \numero{2}, Th. 2]{bourbaki-alg-com-5-7}}]
Let $A$ be an integral domain which is a finitely generated $k$-algebra. Write $K=\Frac(A)$ and let $L/K$ be a finite field extension. Then, the integral closure of $A$ in $L$ is a finitely generated $k$-algebra.
\end{theorem}
By assumption, $R$ is integrally closed, so the integral closure of $R'$ in $K$ coincides with $R$. The above result then implies that $R$ is a finitely generated $k$-algebra. This terminates the proof.
\end{proof}

\subsection{\texorpdfstring{Cox ring of $G$-zips}{}}

We return to the setting of section \ref{glob-zip-ch2}. Let $(G,\mu)$ be a cocharacter datum over $\FF_p$, and write $\GZip^\mu$ for the attached stack of $G$-zips over $k$ (where $k$ is an algebraic closure of $\FF_p$). Although the stack $\GZip^\mu$ is not a $k$-variety (let alone a projective one), we may still define the Cox ring of $\GZip^\mu$ as follows. First, recall that the construction $\lambda\mapsto \Lcal(\lambda)$ of section \ref{sec-pic} induces a homomorphism $X^*(L)\to \Pic(\GZip^\mu)$ with finite cokernel (see \eqref{Pic-zip}). When $\Pic(G)=0$, this map is even surjective. We define the Cox ring of $\GZip^\mu$ as the $k$-algebra
\begin{equation*}
   \Cox(\GZip^\mu) \colonequals \bigoplus_{\lambda\in X^*(L)} H^0(\GZip^\mu, \Lcal(\lambda)).
\end{equation*}
The ring $\Cox(\GZip^\mu)$ is naturally a graded $k$-algebra, whose grading monoid is the set
\begin{equation}
    \Eff(\GZip^\mu) \colonequals \{\lambda\in X^*(L) \ | \ H^0(\GZip^\mu, \Lcal(\lambda))\neq 0\}.
\end{equation}
It can be interpreted geometrically as the set of classes of effective divisors (i.e. the set of $\lambda\in X^*(T)$ such that $\Lcal(\lambda)=\Ocal(D)$ for an effective divisor $D$ on $\GZip^\mu$). It is an additive submonoid of $X^*(L)$ containing $0$. By Proposition \ref{prop-line}, we have:
\begin{equation}
    \Eff(\GZip^\mu) = \{\lambda\in X^*(L) \ | \ \lambda \textrm{ is trivial on }L_\varphi \textrm{ and } \langle \lambda,\delta_\alpha\rangle \geq 0 \ \textrm{for all }\alpha\in \Delta^P \}.
\end{equation}
Next, we investigate the first properties of $\Cox(\GZip^\mu)$. For $\lambda\in X^*(L)$, recall that we may identify an element $f\in H^0(\GZip^\mu,\Lcal(\lambda))$ with a regular map $f\colon G_k\to \AA^1$ satisfying the relation \eqref{Gzip-glob-eq}. We obtain a natural injective ring homomorphism:
\begin{equation}
\xymatrix@M=5pt@C=15pt{
    \Cox(\GZip^\mu)\ar@{^{(}->}[r] & k[G],  \quad (f_\lambda)_{\lambda} \ar@{|->}[r] & \sum_{\lambda\in X^*(L)} f_\lambda.
}
\end{equation}
Hence, we may view the Cox ring of $\GZip^\mu$ as a subalgebra of $k[G]$.

\subsection{First properties}

Fix a Borel pair $(B,T)$ as in section \ref{sec-gp-not}. We consider $k[G]$ as a representation of $T\times T$. Recall that it decomposes as
\begin{equation}
    k[G]=\bigoplus_{(\lambda_1,\lambda_2)\in X^*(T)^2} k[G]_{\lambda_1,\lambda_2}
\end{equation}
where $k[G]_{\lambda_1,\lambda_2}$ is the set of regular maps $f\colon G_k \to \AA^1$ satisfying the relation $f(t_1 g t_2^{-1})=\lambda_1(t_1)\lambda_2(t_2)f(g)$ for all $t_1,t_2\in T$ and $g\in G_k$. For a character $\lambda\in X^*(T)$, define $\SS_\lambda$ as follows:
\begin{equation}
    \SS_\lambda = \{f\colon G_k\to\AA^1 \ | \ f(t g \varphi(t)^{-1})=\lambda(t)f(g), \ t\in T, \ g\in G_k \}.
\end{equation}

\begin{lemma}
For $\lambda\in X^*(T)$, the space $\SS_\lambda$ is the direct sum of the weight spaces $k[G]_{\lambda_1,\lambda_2}$ for pairs $(\lambda_1,\lambda_2)$ satisfying $\lambda_1+ p \sigma^{-1}(\lambda_2) = \lambda$.
\end{lemma}

\begin{proof}
It is immediate to check that $\SS_\lambda$ is stable by the action of $T\times T$ on $k[G]$ (essentially because $T$ is commutative). Therefore, it decomposes as a sum of weight spaces  $k[G]_{\lambda_1,\lambda_2}$. The condition clearly implies that $\lambda_1+ p \sigma^{-1}(\lambda_2) = \lambda$.
\end{proof}

Moreover, we have the decomposition $k[G]=\bigoplus_{\lambda\in X^*(T)} \SS_\lambda$. We define a subalgebra $S\subset k[G]$ as follows. Let $R_{\mathrm{u}}(E')$ denote the unipotent radical of the group $E'$ (see \eqref{Eprime-eq}). It coincides with the set of $(x,y)\in E'$ such that $x\in R_{\mathrm{u}}(B)$. We define $S$ as the subring of function $G_k\to \AA^1$ that are invariant by $R_{\mathrm{u}}(E')$, in other words:
\begin{equation}\label{Sprime-ring}
    S \colonequals \{f\colon G_k\to\AA^1 \ | \ f(x g y^{-1})=f(g), \ (x,y)\in R_{\mathrm{u}}(E'), \ g\in G_k \}.
\end{equation}
Note that $S$ is naturally endowed with an action of $T$. Indeed, if $f\in S$ and $t\in T$, the function $t\cdot f$ defined by $(t\cdot f)(g) \colonequals f(tg\varphi(t)^{-1})$ lies again in $S$, because $(t,\varphi(t))$ (for $t\in T$) normalizes $R_{\mathrm{u}}(E')$. Hence, we can decompose $S$ as follows:
\begin{equation}
    S = \bigoplus_{\lambda\in X^*(T)} S_{\lambda}
\end{equation}
where $S_{\lambda}\subset \SS_\lambda$ is the subspace defined by
\begin{equation}\label{Spr-lam}
    S_\lambda \colonequals \{f\colon G_k\to\AA^1 \ | \ f(x g y^{-1})=\lambda(x)f(g), \ (x,y)\in E', \ g\in G_k \}.
\end{equation}
In the above definition of $S_\lambda$, we view $\lambda$ as a character of $E'$ using the identification $X^*(E')=X^*(B)=X^*(T)$ given by the first projection $\pr_1\colon E'\to B$. When $\lambda\in X^*(L)$, Corollary \ref{cor-Eprime-E} shows that $S_\lambda = H^0(\GZip^\mu,\Lcal(\lambda))$. Therefore, we may write
\begin{equation}
    \Cox(\GZip^\mu)  =  \bigoplus_{\lambda\in X^*(L)} S_\lambda.
\end{equation}
In particular, $\Cox(\GZip^\mu)$ is a subring of $S$. In section \ref{ring-autom-sec}, we will focus our attention to the larger ring $S$, which we will identify as the ring of (vector-valued) automorphic forms on $\GZip^\mu$. Since $\Cox(\GZip^\mu)$ identifies with a subring of $k[G]$, it is an integral domain. Let $K$ be the fraction field of $\Cox(\GZip^\mu)$. By the inclusion into $k[G]$, we may view $K$ as a subfield of $k(G)$.

\begin{lemma}\label{lem-K-Cox}
One has $\Cox(\GZip^\mu) = k[G]\cap K$.
\end{lemma}

\begin{proof}
Since any $f\in \Cox(\GZip^\mu)$ is invariant under the group $R_{\mathrm{u}}(E')$, it is clear that any function $f\in K$ is also invariant under this group. Therefore, any function $f\in k[G]\cap K$ is regular on $G_k$ and $R_{\mathrm{u}}(E')$-invariant on some open subset. By density, it is $R_{\mathrm{u}}(E')$-invariant everywhere, so we deduce $f\in S$. Therefore, $k[G]\cap K\subset S$. Moreover, since $\Cox(\GZip^\mu)$ is a sum of weight spaces of $S$, it is clear that $T$ acts on $\Cox(\GZip^\mu)$, and hence also on $K$. Therefore, we can decompose $k[G]\cap K$ as a sum of $S_\lambda$ for certain characters $\lambda\in X^*(T)$. But if $f\in S_\lambda$ can be written as $f=\frac{h_1}{h_2}$ for $h_1, h_2\in \Cox(\GZip^\mu)$, then by decomposing $h_1$, $h_2$ into homogeneous elements we see that $\lambda$ must lie in $X^*(L)$. Thus $k[G]\cap K=\Cox(\GZip^\mu)$.
\end{proof}

We deduce the following:

\begin{proposition} \label{Cox-first} \ 
\begin{assertionlist}
    \item $\Cox(\GZip^\mu)$ is integrally closed.
    \item The group of units of the ring $\Cox(\GZip^\mu)$ identifies with $k[G]^\times$ (the group of nowhere vanishing regular functions on $G$).
\end{assertionlist}
\end{proposition}

\begin{proof}
Let $f\in K$ and suppose that $f$ is integral over $\Cox(\GZip^\mu)$. Since $G$ is smooth, it is normal and thus $k[G]$ is integrally closed. We deduce that $f\in k[G]\cap K$, and hence $f\in \Cox(\GZip^\mu)$ by Lemma \ref{lem-K-Cox}. Finally, it is clear that any unit of $\Cox(\GZip^\mu)$ lies in $k[G]^\times$. Conversely, any element of $k[G]^\times$ can be written as $f=a\lambda$ for $a\in k^*$ and $\lambda\in X^*(G)$. In particular, $f\in S_{\lambda}$ and the result follows.
\end{proof}

Finally, we determine the transcendence degree of $K$ over $k$. We may define similarly the Cox ring of the open substack $\Ucal_\mu\subset \GZip^\mu$ as follows:
\begin{equation}
    \Cox(\Ucal_\mu) \colonequals \bigoplus_{\lambda\in X^*(L)} H^0(\Ucal_\mu,\Lcal(\lambda)).
\end{equation}
Clearly $\Cox(\GZip^\mu)\subset \Cox(\Ucal_\mu)$. Furthermore, we claim that these two rings share the same field of fractions. Indeed, if $f\in H^0(\Ucal_\mu,\Lcal(\lambda))$, we can multiply $f$ with a large power of a $\mu$-ordinary Hasse invariant to remove the poles along the codimension one strata $Z_{\alpha}$ for $\alpha\in \Delta^P$. This proves the claim.

For each $\lambda\in X^*(L)$, the space $H^0(\Ucal_\mu,\Lcal(\lambda))$ is one-dimensional if $\lambda$ is trivial on the subgroup $L_\varphi$, and is zero otherwise (see \eqref{H0-mu-eq}). Let $N_{\mu}$ denote the order of the finite group scheme $L_\varphi$, so that $N_\mu \lambda$ vanishes on $L_{\varphi}$ for any $\lambda\in X^*(L)$. Hence, the subring of $\Cox(\Ucal_\mu)$ defined as the direct sum of $H^0(\Ucal_\mu,\Lcal(\lambda))$ for $\lambda\in N_\mu X^*(L)$ is isomorphic to $k[X^*(L)]$. Furthermore, the ring extension $k[X^*(L)]\subset \Cox(\Ucal_\mu)$ is integral. In particular, the field of fractions of $\Cox(\Ucal_\mu)$ is a finite extension of that of $k[X^*(L)]$. We deduce the following:

\begin{proposition}\label{cox-trdeg}
    The transcendence degree of $K$ over $k$ is $\rank_{\ZZ}(X^*(L))$.
\end{proposition}

\subsection{Main result}

We now state and prove the main result of this section. We first note that $\Cox(\GZip^\mu)$ satisfies the four assumptions stated before Lemma \ref{lem-equiv-fg}. Indeed, it is an integrally closed integral domain by Proposition \ref{Cox-first}. Since its field of fractions $K$ is contained in $k(G)$, it is finitely generated over $k$. Finally, the space $H^0(\GZip^\mu,\Lcal(\lambda))$ has dimension $\leq 1$ for all $\lambda\in X^*(L)$ by Proposition \ref{prop-H0-dim}. Hence, we may apply Lemma \ref{lem-equiv-fg} in proving our main theorem below.

\begin{theorem}\label{cox-fin-gen-thm}
The ring $\Cox(\GZip^\mu)$ is a finitely generated $k$-algebra.
\end{theorem}

\begin{proof}
By Lemma \ref{lem-equiv-fg}, it suffices to show that $\Eff(\GZip^\mu)_{\QQ_{\geq 0}}$ is finitely generated as a $\QQ_{\geq 0}$-monoid. By Proposition \ref{prop-line}, the set $\Eff(\GZip^\mu)$ coincides with the set of $\lambda\in X^*(L)$ such that $\lambda$ is trivial on $L_\varphi$ and $\langle \lambda, \delta_\alpha \rangle \geq 0$ for all $\alpha\in \Delta^P$. Since $L_\varphi$ is a finite group, we see immediately that:
\begin{equation}
    \Eff(\GZip^\mu)_{\QQ_{\geq 0}} = \{ \lambda\in X^*(L)_\QQ \ | \  \langle \lambda, \delta_\alpha \rangle \geq 0, \ \textrm{for all} \ \alpha\in \Delta^P\}.
\end{equation}
Hence, $\Eff(\GZip^\mu)_{\QQ_{\geq 0}}$ is defined by finitely many inequalities of the form $\ell_i(\lambda)\geq 0$ given by linear forms $\ell_i\colon X^*(L)_\QQ\to \QQ$. Thus $\Eff(\GZip^\mu)_{\QQ_{\geq 0}}$ is a polyhedral cone of $X^*(L)_\QQ$, and in particular is finitely generated. This finishes the proof.
\end{proof}

Applying the usual terminology of projective geometry to this more general setting, we may say that the stack of $G$-zips is a Mori dream space, in view of Theorem \ref{thm-MDS-cox}. Finally, we investigate the simpler case when $G$ has a trivial Picard group. Assume that $\Pic(G)=0$. In this case, for each $\alpha\in \Delta^P$, there exists a map $f_\alpha\colon G\to \AA^1$ such that $\div(f_\alpha)=\overline{Z}_\alpha$. Recall the following lemma:

\begin{lemma}[{\cite[Lemma 5.1.2]{Imai-Koskivirta-partial-Hasse}}]
Let $H$ be a connected smooth algebraic $k$-group and $X$ a $G$-variety. Let $f\colon X\to \AA^1$ be a regular map such that all components of $\div(f)$ are stable by $H$. Then there exists a character $\lambda\in X^*(H)$ such that 
\begin{equation}
    f(h\cdot x)=\lambda(h) f(x)
\end{equation}
for all $h\in H$ and $x\in X$. In particular, $f$ identifies with a section of $\Lcal(\lambda)$ over the quotient stack $[H\backslash X]$.
\end{lemma}

We deduce from this lemma that there exists $\lambda_\alpha\in X^*(L)$ such that $f_\alpha$ is a section of $\Lcal(\lambda_\alpha)$ over $\GZip^\mu$. Since $f_{\alpha}$ only vanishes along $\overline{Z}_\alpha$,  Lemma \ref{sgn-prop} implies that the character $\lambda_\alpha\in X^*(L)$ must satisfy the relations
\begin{equation}
    \langle \lambda_\alpha,\delta_\beta\rangle =0 \quad \textrm{for all} \ \beta\in \Delta^P\setminus\{\alpha\}.
\end{equation}
Recall by Proposition \ref{prop-lin-indep} that the linear forms $\lambda\mapsto \langle \lambda,\delta_\alpha\rangle$ for $\alpha\in \Delta^P$ form a basis of the dual vector space of $\left(X^*(L)/X^*(G)\right)\otimes_\ZZ \QQ$. The above relations show that (up to rescaling), the system $\{\lambda_\alpha\}_{\alpha\in \Delta^P}$ (viewed as a subset of $\left(X^*(L)/X^*(G)\right)\otimes_\ZZ \QQ$) is the dual basis of these linear forms. Write $k[X^*(G)]$ for the subalgebra of $k[G]$ generated by characters $G_k\to \GG_{\mathrm{m}}$.

\begin{proposition}
Assume that $\Pic(G)=0$. The elements $(f_{\alpha})_{\alpha\in \Delta^P}$ are algebraically independent over $k[X^*(G)]$ and we have
\begin{equation}\label{cox-Ha-eq}
    \Cox(\GZip^\mu)=k[X^*(G)][\{f_{\alpha}\}_{\alpha\in \Delta^P}]
\end{equation}
In particular, $\Cox(\GZip^\mu)$ is a polynomial algebra over $k[X^*(G)]$.
\end{proposition}

\begin{proof}
Let $f\in H^0(\GZip^\mu,\Lcal(\lambda))$ for $\lambda\in X^*(L)$. We view $f$ as a regular map $f\in k[G]$. By $E$-equivariance, the divisor $\div(f)$ can be written
\begin{equation}
    \div(f)=\sum_{\alpha\in \Delta^P} m_\alpha \overline{Z}_\alpha
\end{equation}
for some integers $m_\alpha\geq 0$. We deduce that there exists a nowhere vanishing function $f_0\in k[G]^\times$ such that
\begin{equation}
    f=f_0 \prod_{\alpha\in \Delta^P} f_{\alpha}^{m_\alpha}.
\end{equation}
Furthermore $f_0=a\lambda$ for some $a\in k^*$ and $\chi\in X^*(G)$. This shows that \eqref{cox-Ha-eq} holds. Finally, by Proposition \ref{cox-trdeg}, the transcendence degree of $\Cox(\GZip^\mu)$ is the rank of the $\ZZ$-module $X^*(L)$, which is also $\rank_{\ZZ}(X^*(G))+|\Delta^P|$. This shows that $(f_{\alpha})_{\alpha\in \Delta^P}$ are algebraically independent over $k[X^*(G)]$.
\end{proof}

\begin{corollary}\label{cor-semisimple}
Assume that $G$ is semisimple and that $\Pic(G)$ is trivial. In this case, $\Cox(\GZip^\mu)$ is a polynomial algebra over $k$ of dimension $\rank_{\ZZ}(X^*(L))$.
\end{corollary}

\begin{proof}
Since $G$ is semisimple, we have $X^*(G)=0$.
\end{proof}

\section{The ring of automorphic forms} \label{ring-autom-sec}
In this section, we give a geometric interpretation to the larger subring $S\subset k[G]$ introduced in \eqref{Sprime-ring}. We identify it as the Cox ring of the stack of $G$-zip flags.

\subsection{\texorpdfstring{The stack of $G$-zip flags}{}}\label{sec-zip-flag}

The stack of $G$-zip flags was first introduced in \cite{Goldring-Koskivirta-Strata-Hasse} as a group-theoretical analogue of the flag space considered by Ekedahl--van de Geer in \cite{Ekedahl-Geer-EO}. It was used in the construction of generalized Hasse invariants on Ekedahl--Oort strata in \cite{Goldring-Koskivirta-Strata-Hasse}. It can be defined as the quotient stack
\begin{equation}
    \GF^\mu = \left[E\backslash \left( G_k\times P/B \right) \right]
\end{equation}
where $E$ acts on $G_k\times P/B$ by the rule $(x,y)\cdot (g,hB)=(xgy^{-1},xhB)$ for all $(x,y)\in E$, $g\in G_k$, $h\in P$. The first projection $G_k\times P/B\to G$ induces a natural map $\pi\colon \GF^\mu\to \GZip^\mu$ whose fibers are isomorphic to $P/B$. Furthermore, the inclusion $G_k\subset G_k\times P/B$, $g\mapsto (g,1)$ yields an isomorphism of stacks
\begin{equation}\label{isom-GF}
    \GF^\mu \simeq [E'\backslash G_k]
\end{equation}
where $E'$ denotes again the subgroup $E'\colonequals E\cap (B\times G)$ (which already appeared in \eqref{Eprime-eq}). Via the isomorphism \eqref{isom-GF}, the map $\pi$ identifies with the natural projection $[E'\backslash G_k]\to [E\backslash G_k]$. Let $\lambda\in X^*(T)$ be a character. We identify the groups $X^*(E')=X^*(B)=X^*(T)$ via the first projection $\pr_1\colon E'\to B$. We may apply the construction explained in section \ref{sec-vb-quot} to attach to $\lambda$ a line bundle $\Vcal_{\flag}(\lambda)$ on $\GF^\mu$. By \cite[Proposition 3.2.1]{Imai-Koskivirta-partial-Hasse}, one has
\begin{equation}
    \pi_*(\Vcal_{\flag}(\lambda)) = \Vcal_I(\lambda)
\end{equation}
where $\Vcal_I(\lambda)$ is the vector bundle on $\GZip^\mu$ attached to the algebraic $P$-representation
\begin{equation}
    V_I(\lambda)\colonequals \ind_{B}^P(\lambda)
\end{equation}
induced by $\lambda$. Concretely, $V_I(\lambda)$ is the space of regular maps $f\colon P\to \AA^1$ satisfying $f(xb)=\lambda(b)^{-1}f(x)$ for all $b\in B$, $x\in P$. When $\lambda\in X^*(L)$, the line bundle $\Vcal_{\flag}(\lambda)$ coincides by construction with the pullback $\pi^*(\Lcal(\lambda))$. Using \eqref{H0-quot-eq} and the formula $\pi_*(\Vcal_{\flag}(\lambda)) = \Vcal_I(\lambda)$, the space of global section of $\Vcal_{\flag}(\lambda)$ identifies with
\begin{equation}
    H^0(\GF^\mu,\Vcal_{\flag}(\lambda)) = H^0(\GZip^\mu,\Vcal_{I}(\lambda)) =  S_\lambda
\end{equation}
where $S_\lambda$ was defined in \eqref{Spr-lam}. Using the notation introduced in the previous papers \cite{Koskivirta-automforms-GZip} and \cite{Imai-Koskivirta-vector-bundles}, we write
\begin{equation}\label{Rzip-GF}
    R_{\zip}\colonequals \bigoplus_{\lambda\in X^*(T)} H^0(\GZip^\mu,\Vcal_{I}(\lambda)) = \bigoplus_{\lambda\in X^*(T)} H^0(\GF^\mu,\Vcal_{\flag}(\lambda)).
\end{equation}
In \loccitn, this ring was called the ring of automorphic forms on $\GZip^\mu$. This terminology stems from the realm of Shimura varieties (for more details, see section \ref{shim-ring-sec} below).

\begin{definition}\label{cox-flag}
We define the Cox ring of $\GF^\mu$ as the ring $R_{\zip}$.
\end{definition}

By the above discussion, $R_{\zip}$ coincides with the ring $S$ defined in \eqref{Sprime-ring}. In particular, we may view $R_{\zip}$ as a subring of $k[G]$. It contains the Cox ring of $\GZip^\mu$ as a subring.

\begin{proposition} \label{properties-Rzip} \ 
\begin{assertionlist}
    \item We have $R_{\zip}=k[G]\cap K_{\zip}$.
    \item $R_{\zip}$ is integrally closed.
    \item The group of units of $R_{\zip}$ coincides with $k[G]^\times$.
\end{assertionlist}
\end{proposition}

\begin{proof}
To prove (1), recall that $S=R_{\zip}$ is the ring of functions $G_k\to \AA^1$ that are invariant under the group $R_{\mathrm{u}}(E')$. Hence $K_{\zip}$ is the field of $R_{\mathrm{u}}(E')$-invariant functions defined on an open subset of $G_k$. Hence, if $f\in R_{\zip}\cap K$, $f$ is $R_{\mathrm{u}}(E')$-invariant on an open subset, and hence everywhere by density. This shows that $R_{\zip}$ coincides with $k[G]\cap K$. The proof of (2) is completely analogous to that of Proposition \ref{Cox-first}. Finally, since $R_{\zip}$ is a subring of $k[G]$, we have $R_{\zip}^{\times}\subset k[G]^\times$. Conversely, we showed that $k[G]^\times$ is the group of units of $\Cox(\GZip^\mu)$, which is a subring of $R_{\zip}$. Hence $R_{\zip}^{\times}=k[G]^\times$.
\end{proof}

Finally, we determine the field of fractions of $R_\zip$, that we denote by $K_{\zip}\colonequals \Frac(R_{\zip})$. For this, we need to recall the flag stratification of $\GF^\mu$. First, one sees immediately that $E'\subset B\times {}^z B$, where $z=\sigma(w_{0,I})w_0$. Therefore, there is a natural projection map
\begin{equation}
    \psi\colon \GF^\mu \to \left[ B \backslash G_k / {}^z B\right].
\end{equation}
The stack $ \left[ B \backslash G_k / {}^z B\right]$ is a finite stack whose points correspond to the $B\times {}^z B$-orbits in $G_k$ (called the Bruhat strata). They correspond bijectively to the elements of $W$ via the map $w\mapsto F_w\colonequals BwBz^{-1}$. The unique open stratum is $F_{w_0}=Bw_0Bz^{-1}$. For $w\in W$, write $\Fcal_w\colonequals [E'\backslash F_w]$ for the corresponding locally closed substack of $\GF^\mu=[E'\backslash G_k]$. This defines a stratification of $\GF^\mu$ called the flag stratification. We also write $\Ucal_{\max}$ for the unique open stratum. The codimension one strata are $\{\Fcal_{w_0 s_\alpha}\}_{\alpha\in \Delta}$. By \cite{Imai-Koskivirta-partial-Hasse}, there exists for each $\alpha\in \Delta$ a section $\Ha_\alpha\in H^0(\GF^\mu,\Vcal_{\flag}(\ha_\alpha))$ for a certain character $\ha_{\alpha}\in X^*(T)$, such that the vanishing locus of $\Ha_{\alpha}$ is the Zariski closure of $\Fcal_{w_0s_\alpha}$.

We define the Cox ring of $\Ucal_{\rm max}$ as follows:
\begin{equation}
    R_{\rm max} \colonequals \bigoplus_{\lambda\in X^*(T)} H^0(\Ucal_{\rm max},\Vcal_{\flag}(\lambda)).
\end{equation}
In view of \eqref{Rzip-GF}, $R_{\zip}$ is contained in $R_{\rm max}$. Furthermore, if $f\in H^0(\Ucal_{\rm max},\Vcal_{\flag}(\lambda))$, then we can multiply $f$ by large powers of the partial Hasse invariants $\Ha_{\alpha}$ to obtain an element of $R_{\zip}$. This proves that $R_{\zip}$ and $R_{\max}$ have the same field of fractions.

\begin{lemma}
$R_{\max}$ identifies with $k[B\cap M]$.
\end{lemma}

\begin{proof}
By \cite[Lemma 4.2.1 (2)]{Imai-Koskivirta-zip-schubert}, there is an isomorphism $\Ucal_{\max}\simeq [T\backslash (B\cap M)]$ where $T$ acts on $B\cap M$ by $t\cdot b = \varphi(t) b (\sigma(w_{0,I})t^{-1} \sigma(w_{0,I}))$. This implies that $R_{\max}$ can be identified with a subring of $k[B\cap M]$. Furthermore, since the torus $T$ acts on $k[B\cap M]$, this ring decomposes as a direct sum of weight spaces $k[B\cap M]_{\lambda}$. By the isomorphism $\Ucal_{\max}\simeq [T\backslash (B\cap M)]$, the weight space $k[B\cap M]_{\lambda}$ is simply $H^0(\Ucal_{\rm max},\Vcal_{\flag}(\lambda))$. This proves the result.
\end{proof}

We deduce the following:

\begin{proposition}\label{prop-Kzip}
The field $K_{\zip}$ identifies with $k(B\cap M)$.
\end{proposition}

\subsection{Ring of automorphic forms on Shimura varieties}\label{shim-ring-sec}

Let $S_K$ be the mod $p$ special fiber of a Hodge-type Shimura variety with good reduction, as in section \ref{shim-sec}. Let $\zeta\colon S_K\to \GZip^\mu$ be the smooth, surjective morphism explained in \eqref{zeta-eq}. The vector bundle $\zeta^*\Vcal_I(\lambda)$ (denoted again by $\Vcal_I(\lambda)$) is called an automorphic vector bundle on $S_K$. The space $H^0(S_K,\Vcal_I(\lambda))$ is called the space of mod $p$ automorphic forms of weight $\lambda$ and level $K$. There exists also a flag space $\Flag(S_K)$ attached to $S_K$, defined in \cite{Goldring-Koskivirta-Strata-Hasse}. On the special fiber, it can be viewed as the fiber product
$$\xymatrix@1@M=5pt{
\Flag(S_K)\ar[r]^-{\zeta_{\flag}} \ar[d]_{\pi_K} & \GF^\mu \ar[d]^{\pi} \\
S_K \ar[r]_-{\zeta} & \GZip^\mu.
}$$
For each $\lambda\in X^*(T)$, we write again $\Vcal_{\flag}(\lambda)$ for the line bundle on $\Flag(S_K)$ given by pullback via $\zeta_{\flag}$. Since $\zeta$ is smooth, the formula
\begin{equation}
    \pi_{K,*}(\Vcal_{\flag}(\lambda))=\Vcal_I(\lambda)
\end{equation}
also holds on the level of $S_K$. In particular, we have an identification
\begin{equation}\label{ident-flag}
H^0(S_K,\Vcal_I(\lambda))=H^0(\Flag(S_K),\Vcal_{\flag}(\lambda)).
\end{equation}
Similarly to Definition \ref{cox-flag}, we may define a $k$-algebra $R_K$ attached to $S_K$ by
\begin{equation}
    R_K \colonequals \bigoplus_{\lambda\in X^*(T)} H^0(S_K,\Vcal_I(\lambda)).
\end{equation}
Again, via the identification \eqref{ident-flag}, we may interpret $R_K$ as the "automorphic Cox ring" of the scheme $\Flag(S_K)$ (i.e. the subring of the Cox ring of $\Flag(S_K)$ given by the family of line bundles $\{\Vcal_{\flag}(\lambda)\}_{\lambda\in X^*(T)}$). By pullback via $\zeta$, we obtain a natural injective map $\zeta^*\colon R_{\zip}\to R_K$ of graded $k$-algebras.

\subsection{Conjectures}

The following conjecture was first proposed in \cite[Conjecture 5.1.4]{Koskivirta-automforms-GZip}.

\begin{conjecture}\label{Rzip-conj}
The ring $R_{\zip}$ is a finitely generated $k$-algebra. In other words, the stack $\GF^\mu$ is a Mori dream space.
\end{conjecture}

We verify by hand a few cases of this conjecture in section \ref{sec-exa}. For a Shimura variety $S_K$, we do not expect the ring $R_K$ to be finitely generated over $k$. We next present another conjecture relating the rings $R_{\zip}$ and $R_K$. Roughly speaking, it states that the graded algebras $R_{\zip}$ and $R_K$ share the same support (up to multiple). We denote the support of $R_{\zip}$ and $R_K$ as follows:
\begin{align*}
C_{\zip} &\colonequals \{ \lambda\in X^*(T) \ | \ H^0(\GZip^\mu,\Vcal_I(\lambda)) \neq 0\} \\
C_{K} &\colonequals \{ \lambda\in X^*(T) \ | \ H^0(S_K,\Vcal_I(\lambda)) \neq 0\}.   
\end{align*}
Using the terminology of previous articles, we call $C_{\zip}$ the zip cone of $(G,\mu)$. It is an additive submonoid of $X^*(T)$. We have inclusions 
\begin{equation}
    C_{\zip}\subset C_K\subset X^*_{+,I}(T)
\end{equation} 
where $X^*_{+,I}(T)$ denotes the set of $I$-dominant characters. Since $R_{\zip}\subset R_K$, the first inclusion is obvious. Moreover, when $\lambda$ is not $I$-dominant, the induced representation $V_I(\lambda)=\ind_B^P(\lambda)$ is zero, and hence so is the vector bundle $\Vcal_I(\lambda)$. This shows the second inclusion. W. Goldring and the author conjectured in \cite[Conjecture 2.1.6]{Goldring-Koskivirta-global-sections-compositio} the following link between the stack of $G$-zips and the Shimura variety $S_K$:

\begin{conjecture}\label{conj-cone-SK}
We have $C_{\zip,\QQ_{\geq 0}} = C_{K,\QQ_{\geq 0}}$.
\end{conjecture}

Concretely, this conjecture says that if $f\in H^0(S_K,\Vcal_I(\lambda))$ is a nonzero automorphic form of weight $\lambda$, then there exists a nonzero automorphic form $f_{\zip}$ on $\GZip^\mu$ of weight $m\lambda$ for some $m\geq 1$. We list all known cases in the theorem below.

\begin{theorem}\label{thm-cones-6cases}
Conjecture \ref{conj-cone-SK} holds true in all cases below.
\begin{assertionlist}
    \item $S_K$ is a Hilbert--Blumenthal variety.
    \item $S_K$ is a Picard surface at a split prime.
    \item $S_K$ is a Siegel threefold.
    \item $S_K$ is a Siegel-type Shimura variety of rank $3$ and $p\geq 5$.
    \item $S_K$ is a unitary Shimura varieties of signature $(r,s)$ with $r+s\leq 4$. In the case $(r,s)=(2,2)$, assume that $p$ is split.
\item $S_K$ is a Shimura variety of type $A_1$.
\end{assertionlist}
\end{theorem}

In case (6), we say that a Shimura variety of type $A_1$ if $G^{\rm ad}_k$ is a product of copies of $\PGL_{2,k}$ (\cite[\S2.5]{Koskivirta-Hilbert-strata}). In particular, case (6) is a generalization of the case of Hilbert--Blumenthal Shimura varieties (case (1)). W. Goldring and the author proved cases (1), (2) and (3) in \cite[Theorem 4.2.3, Theorem 5.1.1]{Goldring-Koskivirta-global-sections-compositio} and cases (4), (5) in \cite[Theorems 3.4.4, 4.2.7, 4.2.8, 4.3.7]{Goldring-Koskivirta-divisibility}. Finally, case (6) was carried out by the author in \cite[Theorem 5.12.1]{Koskivirta-Hilbert-strata}.

\subsection{Griffiths--Schmid conditions}

The Griffiths--Schmid cone $C_{\GS}$ frequently appears in the theory of automorphic forms (especially in characteristic zero). It is defined as follows:
\begin{equation}
C_{\GS}=\left\{ \lambda\in X^{*}(T) \ \relmiddle| \ 
\parbox{6cm}{
$\langle \lambda, \alpha^\vee \rangle \geq 0 \ \textrm{ for }\alpha\in I, \\
\langle \lambda, \alpha^\vee \rangle \leq 0 \ \textrm{ for }\alpha\in \Phi^+ \setminus \Phi_{L}^{+}$}
\right\}.
\end{equation}
Note that $\lambda$ lies in $C_{\GS}$ if and only if $w_0 w_{0,I}\lambda$ is dominant. W. Goldring and the author proved in \cite[Theorem 2.6.4]{Goldring-Koskivirta-GS-cone} that the weight of any nonzero classical automorphic form in characteristic zero always lies in $C_{\GS}$ (this result was already known to experts, but was not explicitly stated in the literature. We provided a novel, purely characteristic $p$ method). Conversely, if $\lambda\in C_{\GS}$, then it seems to be known to experts that there exists a nonzero automorphic form in characteristic zero of weight $m\lambda$, for some $m>0$.

This cone also plays a role in the theory of $G$-zips. First, it was proved in \cite[Theorem 6.4.3]{Imai-Koskivirta-zip-schubert} that the inclusion $C_{\GS,\QQ_{\geq 0}}\subset C_{\zip,\QQ_{\geq 0}}$ holds for any cocharacter datum $(G,\mu)$. For simplicity, assume now that $P$ is defined over $\FF_p$. In this case, one has $L_{\varphi}=L(\FF_p)$ (see section \ref{glob-subsec}). One has the following result:

\begin{proposition}[{\cite[Proposition 3.7.5]{Koskivirta-automforms-GZip}}]\label{prop-GS}
Assume that $\lambda\in C_{\GS}$.
\begin{assertionlist}
\item There are identifications $H^0(\GZip^\mu,\Vcal_I(\lambda))=H^0(\Ucal_\mu,\Vcal_I(\lambda))=V_I(\lambda)^{L(\FF_p)}$.
\item Any section $f\in H^0(\Ucal_\mu,\Vcal_I(\lambda))$ extends uniquely to $\GZip^\mu$.
\end{assertionlist}
\end{proposition}

Let $\lambda\in C_{\GS}$ be a character. Since $\chi\colonequals w_0w_{0,I}\lambda$ is dominant, we may consider the $G$-representation $V_{\Delta}(w_0w_{0,I}\lambda)$. There is a map of $L$-representations:
\begin{equation}
  \Pi \colon V_\Delta(w_0w_{0,I}\lambda)|_L\to V_I(\lambda)
\end{equation}
defined in \cite[\S 6.3]{Imai-Koskivirta-partial-Hasse}. We recall its definition. Let $f\in V_{\Delta}(\chi)$, viewed as a function $f \colon G_k\to \AA^1$ satisfying $f(xb)=\chi^{-1}(b)f(x)$ for all $x\in G_k$ and $b\in B$. Then $\Pi(f)$ is the function $L\to \AA^1$ mapping $x\in L$ to $f(x w_{0,I}w_0)$. One sees immediately that $\Pi(f)$ lies in $V_I(\lambda)$ and that the map $\Pi$ is $L$-equivariant. Note that the map $\Pi$ is not completely canonical, since it depends on the choice of representatives $w_0$, $w_{0,I}$ made in section \ref{sec-gp-not}. We have the following result:

\begin{proposition}[{\cite[Proposition 6.3.1]{Imai-Koskivirta-partial-Hasse}}] \label{prop-PiX-isom}
The map $\Pi$ induces an isomorphism of $L$-representations 
\begin{equation}
V_{\Delta}(w_0w_{0,I}\lambda)^{R_{\mathrm{u}}(P)} \to V_I(\lambda).    
\end{equation}
In particular, $V_{\Delta}(w_0w_{0,I}\lambda)|_L$ decomposes as $V_{\Delta}(w_0w_{0,I}\lambda)|_L=V_I(\lambda) \oplus \Ker(\Pi)$.
\end{proposition}
This implies also that $\Pi$ is surjective. Recall the definition of $\Cox_I(G)$, $\Cox_\Delta(G)$ from \eqref{coxI-def-eq}. Taking direct sums, the map $\Pi$ induces a homomorphism of $k$-algebras:
\begin{equation}\label{cox-Pi}
   \Pi\colon \Cox_{\Delta}(G) = \bigoplus_{\lambda\in X^*(T)} V_{\Delta}(w_0w_{0,I}\lambda) \longrightarrow  \Cox_I(G) = \bigoplus_{\lambda\in X^*(T)} V_{I}(\lambda) 
\end{equation}
Note that this ring homomorphism does not preserve the grading as it sends the $w_0w_{0,I}\lambda$-graded piece into the $\lambda$-graded piece. Furthermore, the morphism $\Pi$ in \eqref{cox-Pi} is no longer surjective, because only characters $\lambda$ such that $w_0w_{0,I}\lambda$ is dominant appear in the image. In other words, the image of the ring homomorphism $\Pi$ is precisely the subalgebra
\begin{equation}
    \Cox_{I}^{\GS}(G) \colonequals \bigoplus_{\lambda\in C_{\GS}} V_{I}(\lambda).
\end{equation}
In particular, since $\Cox_{\Delta}(G)$ is finitely generated, we deduce the following:
\begin{corollary}
The $k$-algebra $ \Cox_{I}^{\GS}(G)$ is finitely generated.
\end{corollary}
Now, we return to the ring of automorphic forms $R_{\zip}$. Given the above discussion, it is natural to investigate the subalgebra $R_{\GS}\subset R_{\zip}$ defined as follows:
\begin{equation}
    R_{\GS}\colonequals \bigoplus_{\lambda\in C_{\GS}} H^0(\GZip^\mu,\Vcal_I(\lambda)).
\end{equation}
We continue to assume that $P$ is defined over $\FF_p$. By Proposition \ref{prop-GS}, for all $\lambda\in C_{\GS}$ we have $H^0(\GZip^\mu,\Vcal_I(\lambda))=V_I(\lambda)^{L(\FF_p)}$. Hence, we obtain
\begin{equation}
    R_{\GS}=\bigoplus_{\lambda\in C_{\GS}} V_I(\lambda)^{L(\FF_p)} = \left( \bigoplus_{\lambda\in C_{\GS}} V_I(\lambda)\right)^{L(\FF_p)} = (\Cox_{I}^{\GS}(G))^{L(\FF_p)}.
\end{equation}
Since $\Cox_{I}^{\GS}(G)$ is finitely generated and $L(\FF_p)$ is finite, we deduce by Emmy Noether's theorem (\cite{Noether-fin-gen}) the following partial result:
\begin{proposition}\label{RGS-fg-prop}
    The subalgebra $R_{\GS}$ is finitely generated.
\end{proposition}
More generally, let $\Gamma\subset C_{\zip}$ be a subcone. Then we may define a subalgebra
\begin{equation}
    R_{\Gamma}\colonequals \bigoplus_{\lambda\in \Gamma} H^0(\GZip^\mu,\Vcal_I(\lambda)).
\end{equation}
When $\Gamma\subset C_{\zip}$ and $\Gamma_{\RR_{\geq 0}}$ has a nonempty interior in $C_{\zip,\RR_{\geq 0}}$, then $R_\Gamma$ has the same field of fractions as $R_{\zip}$. Indeed, fix any section $h_0\in H^0(\GZip^\mu,\Vcal_I(\lambda_0))$ for $\lambda_0$ in the interior of $\Gamma_{\RR_{\geq 0}}$. If $f\in H^0(\GZip^\mu,\Vcal_I(\lambda))$ is any section, then $fh_0^m$, lies in $R_{\Gamma}$ for large $m$. This implies the claim. In particular, $R_{\GS}$ and $R_{\zip}$ have the same field of fractions.

\subsection{\texorpdfstring{Direct products over $\FF_p$}{}}

We consider two cocharacter data $(G_1,\mu_1)$ and $(G_2,\mu_2)$ (see section \ref{zip-def}) and define $G=G_1\times G_2$ (product taken over $\FF_p$), endowed with the cocharacter $\mu=(\mu_1,\mu_2)$. Choose Borel pairs $(B_i,T_i)$ for $G_i$ ($i\in \{1,2\}$) as in section \ref{sec-gp-not} and define $B \colonequals B_1\times B_2$, $T \colonequals T_1\times T_2$. Let $R^{(1)}_{\zip}$ and $R_{\zip}^{(2)}$ denote the ring of automorphic forms of $(G_1,\mu_1)$ and $(G_2,\mu_2)$ respectively, and write $R_{\zip}$ for the ring of $(G,\mu)$. Similarly, let also $C_{\zip}^{(1)}$, $C_{\zip}^{(2)}$ and $C_{\zip}$ denote the associated zip cones. As explained in section \ref{sec-zip-flag}, we may identify $R_{\zip}$ with a subring of $k[G]=k[G_1]\otimes_k k[G_2]$. Furthermore, we have $X^*(T)=X^*(T_1)\times X^*(T_2)$.

\begin{proposition}\label{prop-Fp-prod}
Via this identification, one has 
\begin{align*}
    R_{\zip} &= R_{\zip}^{(1)} \otimes_k R_{\zip}^{(2)}, \\
    C_{\zip} & = C_{\zip}^{(1)}\times C_{\zip}^{(2)}.
\end{align*}
\end{proposition}

\subsection{Quotient by a central subgroup}

Let $(G,\mu)$ be a cocharacter datum, and let $\Zcal_\mu=(G,P,Q,L,M)$ be the associated zip datum. Choose a Borel pair $(B,T)$ as in section \ref{sec-gp-not} and fix a subgroup $Z\subset T$ contained in the center of $G$ and defined over $\FF_p$. Let $G'\colonequals G/Z$ and define $\mu'\colonequals \pi\circ \mu$ where $\pi\colon G\to G'$ is the natural projection map. Recall that the formation of the stack of $G$-zips is functorial in the pair $(G,\mu)$ (\cite{Goldring-Koskivirta-zip-flags}), in the following sense: Let $(G_1,\mu_1)$, $(G_2,\mu_2)$ be cocharacter data and $f\colon G_1\to G_2$ a group homomorphism defined over $\FF_p$ such that $\mu_2=f\circ \mu_1$ over $k$. Let $\varpi_i\colon G_i \to \GiZip^{\mu_i}$ (for $i=1,2$) be the natural quotient map. Then $f$ induces a natural morphism of stacks
\begin{equation}
    \widetilde{f}\colon \GoneZip^{\mu_1} \to \GtwoZip^{\mu_2}
\end{equation}
such that the natural diagram commutes, i.e. one has $\widetilde{f} \circ \varpi_1 = \varpi_2 \circ f$. From this, we obtain a natural morphism of stacks $\widetilde{\pi}\colon \GZip^{\mu}\to \GpZip^{\mu'}$. Since $Z$ is a central subgroup, one sees easily that $\widetilde{\pi}$ is a homeomorphism on the underlying topological spaces. Define $T'\colonequals T/Z$ and $B'\colonequals B/Z$. For $\lambda'\in X^*(T')$, write again $\lambda'$ for the character $\lambda'\circ \pi \in X^*(T)$. In this way, view $X^*(T')$ as a subgroup of $X^*(T)$. We have a natural pullback map
\begin{equation}
    \pi^*\colon H^0(\GpZip^{\mu'},\Vcal_{I}(\lambda')) \to H^0(\GZip^{\mu},\Vcal_{I}(\lambda'))
\end{equation}
that is clearly injective. Furthermore, since $Z$ is a central subgroup, it acts naturally on the space $H^0(\GZip^{\mu},\Vcal_{I}(\lambda))$ for any $\lambda\in X^*(T)$. Indeed, let $f\in H^0(\GZip^{\mu},\Vcal_{I}(\lambda))$, viewed as a regular map $f\colon G_k\to \AA^1$, satisfying $f(xgy^{-1})=\lambda(x)f(g)$ for all $(x,y)\in E'$, $g\in G_k$. For $z\in Z$, define the function
$z\cdot f \colon G_k\to \AA^1$ by $g\mapsto f(gz)$. One sees immediately that it lies again in $H^0(\GZip^{\mu},\Vcal_{I}(\lambda))$. Write $H^0(\GZip^{\mu},\Vcal_{I}(\lambda'))^{Z}$ for the subspace of $Z$-invariant elements.

\begin{lemma} \ 
\begin{assertionlist}
    \item Assume that $H^0(\GZip^{\mu},\Vcal_{I}(\lambda))^Z$ is nonzero. Then $\lambda$ is trivial on $Z$, hence lies in the subgroup $X^*(T')\subset X^*(T)$.
    \item If $\lambda'\in X^*(T')$, the map $\pi^*$ induces an identification
\begin{equation}
    H^0(\GpZip^{\mu'},\Vcal_{I}(\lambda')) = H^0(\GZip^{\mu},\Vcal_{I}(\lambda'))^{Z}.
\end{equation}
\end{assertionlist}
\end{lemma}

\begin{proof}
Let $f\colon G_k\to \AA^1$ be a nonzero, $Z$-equivariant function satisfying $f(xgy^{-1})=\lambda(x)f(g)$ for all $(x,y)\in E'$, $g\in G_k$. Since $Z$ is defined over $\FF_p$, we have $f(zg\varphi(z)^{-1})=f(z\varphi(z)^{-1}g)=f(g)=\lambda(z)f(g)$ for all $z\in Z$ and $g\in G_k$. Hence $\lambda$ is trivial on $Z$, which proves (1). The second assertion is immediate.
\end{proof}

Write $R_{\zip}$ and $R'_{\zip}$ for the rings of $(G,\mu)$ and $(G',\mu')$ respectively. Assertions (1) and (2) above have the following consequence.

\begin{corollary}\label{cor-Rpzip}
Pullback by $\pi$ induces an identification $R'_{\zip} = R_{\zip}^Z$.
\end{corollary}

Write $Z_G$ for the center of $G$. The connected component $Z^\circ_G$ is a subtorus. Recall that the character group of $Z_G$ identifies with $X^*(T)/\ZZ \Phi$, where $\ZZ \Phi$ is the subgroup generated by $\Phi$. In particular, the map $X^*(G)\to X^*(Z_G)$ has finite cokernel. Similarly, the inclusion $Z\subset Z_G$ induces a map $X^*(Z_G)\to X^*(Z)$ which has finite cokernel. This implies that there exists $N\geq 1$ such that for any $\lambda\in X^*(T)$, there exists a character $\chi\in X^*(G)$ satisfying $(N\lambda)|_Z = \chi|_Z$. Using this, we prove the following:

\begin{proposition} \ \label{Rprime-vs-R}
    \begin{assertionlist}
        \item The ring extension $R'_{\zip}\left[ X^*(G) \right] \subset R_{\zip}$ is integral.
\item $R_{\zip}$ is a finitely generated $k$-algebra if and only if $R'_{\zip}$ is a finitely generated $k$-algebra.
\end{assertionlist}
\end{proposition}

\begin{proof}
Let $f\in H^0(\GZip^\mu,\Vcal_I(\lambda))$, viewed as a regular map $f\colon G_k\to \AA^1$ satisfying $f(xgy^{-1})=\lambda(x)f(g)$ for all $(x,y)\in E'$, $g\in G_k$. By the above discussion, there exists $N\geq 1$ such that $(N\lambda)|_Z = \chi|_Z$ for some $\chi\in X^*(G)$. Consider the function
\begin{equation}
    h\colon G_k\to \AA^1, \quad g\mapsto \chi(g)^{-1} f(g)^N.
\end{equation}
It is clear that $h$ is a nonzero element of $H^0(\GZip^\mu,\Vcal_I(N\lambda-\chi))$. Since $N\lambda-\chi$ is trivial on $Z$, we deduce from Corollary \ref{cor-Rpzip} that $h\in R'_{\zip}$. In particular, $f^N$ lies in the subring $R'_{\zip}\left[ X^*(G) \right]$, which proves (1). Now, assume that $R_{\zip}$ is a finitely generated $k$-algebra. Since $Z$ is reductive (not necessarily connected), we deduce that the invariant algebra $R'_{\zip}=R_{\zip}^Z$ is finitely generated by Nagata's theorem. Conversely, assume that $R'_{\zip}$ is finitely generated. By (1), $R_{\zip}$ is a integral extension of a finitely generated $k$-algebra. Since $R_{\zip}$ is integrally closed, we deduce by a similar argument as in the proof of Lemma \ref{lem-equiv-fg} that $R_{\zip}$ is finitely generated. This terminates the proof.
\end{proof}

\subsection{General results}

We note that Conjecture \ref{Rzip-conj} admits the following weaker variants. In increasing order of difficulty, we conjecture the following statements:
\medskip
\begin{Alist}
    \item $C_{\zip,\QQ_{\geq 0}}$ is a finitely generated $\QQ_{\geq 0}$-monoid.
    \item $C_{\zip}$ is a finitely generated monoid.
    \item $R_{\zip}$ is a finitely generated $k$-algebra.
\end{Alist}
\medskip
It is clear that (C) implies (B), and (B) implies (A). Furthermore, Proposition \ref{prop-Fp-prod} implies that if any of the statements (A), (B), (C) holds for two cocharacter data $(G_1,\mu_1)$ and $(G_2,\mu_2)$, then it also holds for the product (over $\FF_p$). We first state known results about property (A). Recall the following definition:

\begin{definition}
Let $(G,\mu)$ be a cocharacter datum and let $\Zcal_\mu=(G,P,Q,L,M)$ be the attached zip datum. We say that $(G,\mu)$ is of Hasse-type if  $L$ is defined over $\FF_p$ and $\sigma$ acts on $I$ by $-w_{0,I}$.
\end{definition}

\begin{exa}
The following cases are of Hasse-type
\begin{bulletlist}
    \item $G$ is an odd orthogonal group and $\mu$ is minuscule.
    \item  $G=\GL_{3,\FF_p}$ and $L$ a Levi subgroup of type $(2,1)$ or $(1,2)$.
    \item $G=\GSp_{4,\FF_p}$ and $\mu$ is minuscule.
\end{bulletlist}  
\end{exa}
A complete classification of Hasse-type cocharacter data is carried out in the appendix (by W. Goldring) of the article \cite{Imai-Koskivirta-zip-schubert} of N. Imai and the author. By \cite[Theorem 4.3.1]{Imai-Koskivirta-zip-schubert}, the condition of being of Hasse-type is equivalent to the equality $C_{\pha,\QQ_{\geq 0}} = C_{\zip,\QQ_{\geq 0}}$, where $C_{\pha}$ is the cone of partial Hasse invariants, defined as the image of $X^*_+(T)$ by the linear map
\begin{equation}\label{equ-maph}
h\colon X^*(T)\to X^*(T), \quad \lambda \mapsto \lambda- p w_{0,I} (\sigma^{-1} \lambda).
\end{equation}
The cone $C_{\pha}$ can be interpreted as the cone spanned by the weights of partial Hasse invariants (\loccitn, \S 3.6). It is clear that $C_{\pha,\QQ_{\geq 0}}$ is always finitely generated, so we deduce that if $(G,\mu)$ is a cocharacter datum of Hasse-type, then (A) holds true. We record in the proposition below other known cases:

\begin{proposition}
Statement \textup{(A)} holds true in each of the cases below:
\begin{assertionlist}
\item $(G,\mu)$ is of Hasse-type.
\item $(G,\mu)$ corresponds to any of the cases \textup{(1)} through \textup{(6)} of Theorem \ref{thm-cones-6cases} (without restriction on $p$ in case \textup{(4)}).
\end{assertionlist}
\end{proposition}

Next, we move on to properties (B) and (C), which are much stronger than (A). Before we discuss known cases, we point out that these properties hold in the following trivial cases:
\begin{proposition} \label{PBPG}
If $P=B$ or $P=G$, then property \textup{(C)} holds true.
\end{proposition}

\begin{proof}
In case $P=B$, the stack $\GF^\mu$ coincides with $\GZip^\mu$, and the ring $R_{\zip}$ is simply the ring $\Cox(\GZip^\mu)$, which is finitely generated by Theorem \ref{cox-fin-gen-thm}. Next, assume that $P=G$. In this case, we have $C_{\GS}=X^*_{+,I}(T)$, hence $C_{\zip}\subset C_{\GS}$ and hence $R_{\zip}=R_{\GS}$, which is finitely generated by Proposition \ref{RGS-fg-prop}.
\end{proof}

Apart from these trivial cases, property (C) has only been checked in the following situation: Let $G$ be the symplectic group $\Sp_4$ given by the alternating matrix
\begin{equation}
\Psi:=\left(\begin{matrix}
& -J \\
J&\end{matrix}\right)
\quad \textrm{ where } \quad
J:=\left(\begin{matrix}
&1 \\
1&
\end{matrix}\right).
\end{equation} 
Let $\mu$ be the cocharacter $z\mapsto \diag(z,z,z^{-1},z^{-1})$. We showed the following:
\begin{proposition}[{\cite[Theorem 5.4.1]{Koskivirta-automforms-GZip}}] \label{sp4-Rzip}
In this case, $R_{\zip}$ is a polynomial algebra in three variables.
\end{proposition}
By Proposition \ref{Rprime-vs-R}, property (C) also holds for the group $\GSp_{4,\FF_p}$, which corresponds to the case of Siegel-type Shimura varieties of rank $2$.

\subsection{\texorpdfstring{Unitary groups of rank $3$}{}}\label{sec-exa}
We prove Conjecture \ref{Rzip-conj} in the case of unitary groups in three variables.

\subsubsection{Unitary groups and unitary Shimura varieties}

We first give some general notation about unitary groups. Let $\mathbf{E}/\QQ$ be a totally imaginary quadratic extension, and fix a hermitian space $(\mathbf{V},\psi)$ of dimension $3$ over $\mathbf{E}$. We assume for simplicity that there exists a basis $\Bcal$ in which $\psi$ is given by the following matrix:
\begin{equation}\label{mat-anti}
    \begin{pmatrix}
    &&1\\&1&\\1&&
    \end{pmatrix}.
\end{equation}
Let $\mathbf{G}=\GU(\mathbf{V},\psi)$ be the general unitary group of $(\mathbf{V},\psi)$. Let $(r,s)$ denote the signature of
$\psi_\RR$, where $r,s$ are nonnegative integers such that $r+s=3$ (the most interesting case for us is when $(r,s)=(2,1)$ or $(1,2)$ and corresponds to Shimura varieties of dimension two called Picard surfaces). Fix a prime $p$ that is unramified in $\mathbf{E}$ and let
$\Lambda\subset \mathbf{V}\otimes_{\QQ} \QQ_p$ be the $\Ocal_{\mathbf{E}}\otimes_{\ZZ}\ZZ_p$-lattice generated by the elements of $\Bcal$. The $\ZZ_p$-group scheme $\Gcal\colonequals \GU(\Lambda,\psi)$ gives a reductive $\ZZ_p$-model of $\mathbf{G}_{\QQ_p}$. Write $K_p=\Gcal(\ZZ_p)$ for the associated hyperspecial subgroup of $\mathbf{G}(\QQ_p)$.

Let $K^p\subset \mathbf{G}(\AA^p_f)$ be a sufficiently small open compact subgroup and $K\colonequals K_pK^p$. Kottwitz has constructed in \cite{Kottwitz-points-shimura-varieties} a PEL-type Shimura variety $\Sscr_{K}$ over $\Ocal_{\mathbf{E}_v}$ which is a smooth, canonical model of the Shimura variety $\Sh_K(\mathbf{G},\mathbf{X})$. Write $S_K=\Sscr_K\otimes_{\Ocal_{\mathbf{E}_v}} k$ for its special fiber and $G\colonequals \Gcal\otimes_{\ZZ_p} \FF_p$. Recall that there is a smooth surjective map $\zeta\colon S_K\to \GZip^{\mu}$, where $\mu$ is naturally attached to the Shimura datum. The group $G$ depends on the ramification of $p$:
\begin{enumerate}[(1)]
    \item If $p$ is split in $\mathbf{E}$, then $G$ is isomorphic to $\GL_{3,\FF_p}\times \GG_{\mathrm{m},\FF_p}$. For simplicity, we will treat the case of $G=\GL_{3,\FF_p}$.
    \item If $p$ is inert in $\mathbf{E}$, then $G$ is a general unitary group $\GU(3)$ over $\FF_p$. For simplicity, we will consider the usual unitary group $G=\U(3)$.
\end{enumerate}

We prove Conjecture \ref{Rzip-conj} for the group $G$ in the two cases explained above. We will see that in case (1), Conjecture \ref{Rzip-conj} can be reduced to the same statement for the symplectic group $G'=\Sp_{4,\FF_p}$, which is already known by Proposition \ref{sp4-Rzip}. Case (2) will require more work.

\subsubsection{The split case}\label{split-unitary-sec}

We first recall below how the stack of $\GL_n$-zips (for a cocharacter of type $(n-1,1)$) is related to the stack of $G'$-zip for the symplectic group $G'=\Sp_{2(n-1),\FF_p}$. For the time being, we let $n\geq 2$ be any integer. We specialize to the case $n=3$ later. Put $G=\GL_{n,\FF_p}$ and consider the cocharacter
\begin{equation}\label{mu-unitary}
   \mu \colon \GG_{\mathrm{m},k}\to \GL_{n,k}, \quad z\mapsto 
    \begin{pmatrix}
    z&&&\\&\ddots&&\\&&z& \\ &&&1
    \end{pmatrix}.
\end{equation}
Write $\Zcal_{\mu}=(G,P,Q,L,M)$ for the attached zip datum (since $P$ is defined over $\FF_p$, note that $L=M$ in this case). Denote by $(u_1, \dots ,u_n)$ the canonical basis of $k^n$. The parabolic $P$ is the stabilizer of the line $ku_n$ and $Q$ is the stabilizer of $\Span_k(u_{1},\dots, u_{n-1})$. Let $(B,T)$ be the Borel pair consisting of the lower-triangular Borel subgroup $B$ and the diagonal torus $T$. The Levi subgroup $L=P\cap Q$ is defined over $\FF_p$ and is isomorphic to $\GL_{n-1,\FF_p}\times \GG_{\mathrm{m},\FF_p}$. We make the identification $X^*(T)=\ZZ^n$ such that $(a_1,\dots,a_n)\in \ZZ^3$ corresponds to the character $\diag(x_1, \dots ,x_n)\mapsto \prod_{i=1}^n x_i^{a_i}$. Write $(e_1,\dots ,e_n)$ for the standard basis of $\ZZ^n$. The simple roots of $G$ are $\Delta=\{\alpha_1,\dots,\alpha_{n-1}\}$ where $\alpha_i= e_i-e_{i+1}$ ($i=1,\dots, n-1$). The cones $X_{+,I}^*(T)$ and $\Ccal_{\GS}$ are given by:
\begin{align*}
     X_{+,I}^*(T) &= \{(a_1,\dots,a_n)\in \ZZ^n \mid a_1\geq a_2 \dots \geq a_{n-1} \} \\
     \Ccal_{\GS} &= \{(a_1, \dots ,a_n)\in X_{+,I}^*(T) \mid a_1\leq a_n \}.
\end{align*}
On the other hand, consider the group $G'\colonequals \Sp_{2(n-1),\FF_p}$, which is the symplectic group attached to the alternating matrix
\begin{equation}
\Psi:=\left(\begin{matrix}
& -J \\
J&\end{matrix}\right)
\quad \textrm{ where } \quad
J:=\left(\begin{matrix}
&&1 \\
&\iddots& \\
1&&
\end{matrix}\right)
\end{equation} 
and the size of the matrix $J$ is $(n-1)\times (n-1)$. Write $T'$ for the maximal torus consisting of diagonal matrices in $G'$, and $B'$ for set of lower-triangular matrices in $G'$ (one shows easily that $B'$ is a Borel subrgoup of $G'$). For any $k$-algebra $A$, one has
\begin{equation}
    T'(A)=\{\diag(t_1,\dots,t_{n-1},t_{n-1}^{-1},\dots, t_1^{-1}) \ | \ t_1,\dots, t_n\in A^\times\}.
\end{equation}
Consider the cocharacter $\mu'\colon \GG_{\mathrm{m},k}\to G'_k$, $z\mapsto \diag(z I_{n-1},z^{-1} I_{n-1})$, and write $\Zcal'=(G',P',Q',L',M')$ and $\GpZip^{\mu'}$ for the attached zip datum and stack of $G'$-zips respectively. The Levi subgroup $L'$ is isomorphic to $\GL_{n-1,\FF_p}$. Identify $X^*(T')=\ZZ^{n-1}$ such that $(a_1,\dots, a_{n-1})\in \ZZ^{n-1}$ corresponds to the character that maps $\diag(t_1,\dots,t_{n-1},t_{n-1}^{-1},\dots, t_1^{-1})$ to $\prod_{i=1}^{n-1} t_i^{a_i}$.

We now explain the relation between the stacks $\GZip^\mu$ and $\GpZip^{\mu'}$. For any $\lambda\in \ZZ^n$ (resp. $\lambda'\in \ZZ^{n-1}$), write $\Vcal_I(\lambda)$ (resp. $\Vcal_{I'}(\lambda')$) for the associated vector bundle on $\GZip^\mu$ (resp. $\GpZip^{\mu'}$), as explained in \ref{sec-vb-quot}.

\begin{proposition}[{\cite[\S 4.2.2]{Goldring-Koskivirta-divisibility}}] \label{prop-corresp} \ 
Let $\lambda\in \ZZ^n$ of the form $\lambda =(\lambda_1,\dots, \lambda_{n-1},0)$ and set $\overline{\lambda}\colonequals (\lambda_1,\dots,\lambda_{n-1})$. Then, there is an identification
\begin{equation}
   H^0(\GZip^{\mu},\Vcal_I(\lambda)) = H^0(\GpZip^{\mu'},\Vcal_{I'}(\overline{\lambda})). 
\end{equation}
\end{proposition}

Write $R_{\zip}$ and $R'_{\zip}$ for the rings of automorphic forms with respect to $(G,\mu)$ and $(G',\mu')$ respectively. Let $R_{\zip,0}\subset R_{\zip}$ denote the subring
\begin{equation}
   R_{\zip,0} \colonequals \bigoplus_{\lambda=(\lambda_1,\dots,\lambda_{n-1},0)\in \ZZ^n} H^0(\GZip^\mu,\Vcal_I(\lambda)). 
\end{equation}
From the above proposition, we deduce that there is an identification $R_{\zip,0}=R'_{\zip}$. Furthermore, note that the determinant function $\det\colon G_k \to \GG_{\mathrm{m},k}$ is a non-vanishing section of $\Lcal(\lambda_{\det})$ on $\GZip^\mu$, where $\lambda_{\det}=-(p-1,\dots,p-1)$.

\begin{proposition}
One has $R_{\zip}=R_{\zip,0}[\det,\det^{-1}]$.
\end{proposition}

\begin{proof}
Write $L_{\varphi}\subset L$ for the group \ref{Lvarphi-eq} attached to $(G,\mu)$. Since $P$ is defined over $\FF_p$, we have $L_\varphi = L(\FF_p)\simeq \GL_{n-1}(\FF_p)\times \FF_p^{\times}$. For $\lambda\in \ZZ^n$, the space $H^0(\GZip^\mu,\Vcal_I(\lambda))$ is in particular always contained in $V_I(\lambda)^{L_\varphi}$ by Theorem \ref{thm-main-H0}. Therefore, we deduce that $H^0(\GZip^\mu,\Vcal_I(\lambda))=0$ if $\lambda_n$ is not a multiple of $p-1$. Moreover, since $\det$ is non-vanishing, multiplication by $\det$ induces an isomorphism 
\begin{equation}
    H^0(\GZip^\mu,\Vcal_I(\lambda))\to H^0(\GZip^\mu,\Vcal_I(\lambda+\lambda_{\det})).
\end{equation}
It follows that we can always shift the weight so that $\lambda_n=0$. In this case, $H^0(\GZip^\mu,\Vcal_I(\lambda))$ is contained in the subalgebra $R_{\zip,0}$, which terminates the proof.
\end{proof}

When $n=3$, we already know that $R'_{\zip}$ is finitely generated (Proposition \ref{sp4-Rzip}). Therfore, we obtain the following:

\begin{corollary}\label{cor-split-GL3}
For $n=3$, the $k$-algebra $R_{\zip}$ is finitely generated. It is isomorphic to $k[a,b,c,d,d^{-1}]$ where $a,b,c,d$ are indeterminates.
\end{corollary}

The above corollary settles the case of a parabolic of type $(r,s)=(2,1)$ for the group $G=\GL_{3,\FF_p}$. The case $(r,s)=(1,2)$ is completely symmetric and can be treated similarly. Finally, when $P=G$ or $P=B$, the result holds by Proposition \ref{PBPG}.

\subsubsection{The inert case}

Next, we consider the more difficult case when $G$ is a unitary group over $\FF_p$. Let $(V,\psi)$ be a $3$-dimensional $\FF_{p^2}$-vector space endowed with a non-degenerate hermitian form $\psi \colon V\times V\to \FF_{p^2}$. We assume again that there is a basis $\Bcal$ of $V$ where $\psi$ is given by the matrix \eqref{mat-anti}. Let $G=\U(V,\psi)$ be the associated unitary group. Set $\Gal(\FF_{p^2}/\FF_p)=\{\id, \sigma\}$. For any $\FF_{p^2}$-algebra $A$, there is an isomorphism $\FF_{p^2}\otimes_{\FF_p} A\to A\times A$, $u\otimes x\mapsto (ux,\sigma(u)x)$. This induces an isomorphism $G_{\FF_{p^2}}\simeq \GL(V)$. Using the basis $\Bcal$, we identify $\GL(V)$ and $\GL_{3,\FF_{p^2}}$. The action of $\sigma$ on the set $\GL_3(k)$ is given by $\sigma\cdot A = J ({}^t \sigma(A)^{-1})J$ (where $\sigma(A)$ denotes the usual Frobenius action on $\GL_3(k)$). Let $T$ denote the maximal diagonal torus of $\GL_{3}$ and $B$ the lower-triangular Borel subgroup. By our choice of the basis $\Bcal$, the groups $B$ and $T$ are defined over $\FF_p$. Identify $X^*(T)=\ZZ^3$ as in section \ref{split-unitary-sec}. We let again $\mu\colon \GG_{\mathrm{m},k}\to \GL_{3,k}$ denote the cocharacter \eqref{mu-unitary} for $n=3$. Let $\Zcal_{\mu}=(G,P,Q,L,M,\varphi)$ be the associated zip datum. In this case, note that $P$ is no longer defined over $\FF_p$. The determinant function $\det  \colon  G_k\to \GG_{\mathrm{m}}$ is now a section of weight $\lambda_{\det}=(p+1, p+1, p+1)$. We recall the expression for the space $H^0(\GZip^\mu,\mathcal{V}_I (\lambda))$ determined in \cite{Imai-Koskivirta-vector-bundles}. Let $\lambda=(\lambda_1,\lambda_2,\lambda_3)\in X^*_{+,I}(T)$. Under the isomorphism 
\[
 \GL_{2}\times \GG_{\mathrm{m}} \to L;\ (A,z) \mapsto 
 \left( \begin{matrix}
 A &\\ & z
 \end{matrix} \right), 
\]
the representation $V_I(\lambda)$ corresponds to the $\GL_{2}\times \GG_{\mathrm{m}}$-representation
\[ 
 \left(\deter_{\GL_2}^{\lambda_2} \otimes \Sym^{\lambda_1-\lambda_2}(\Std_{\GL_{2}})\right) \boxtimes \xi_{\lambda_3} 
\]
where $\Std_{\GL_{2}}$ is the standard representation of $\GL_2$ and $\xi_r$ is the character $z\mapsto z^r$ of $\GG_{\mathrm{m}}$. The representation $V_I(\lambda)$ has dimension $\lambda_1-\lambda_2+1$ and its weights are given by
\[
 \nu_i \colonequals (\lambda_1-i,\lambda_2+i,\lambda_3), \quad 0\leq i \leq \lambda_1-\lambda_2. 
 \]
Write $V_I(\lambda)_{\nu_i}$ for the weight space corresponding to $\nu_i$. For $\lambda\in \ZZ^3$, define:
\begin{equation}
     F(\lambda) = \frac{p}{p^2-p+1}
 (p\lambda_1-(p-1)\lambda_2-\lambda_3). 
\end{equation}
We showed:
\begin{proposition}[{\cite[Proposition 6.3.2]{Imai-Koskivirta-vector-bundles}}]
For $\lambda=(\lambda_1,\lambda_2,\lambda_3)\in X^*_{+,I}(T)$, we have 
\begin{equation}\label{equ-propH0}
 H^0(\GZip^\mu,\mathcal{V}_I (\lambda)) = \bigoplus_{\substack{p|i, \ p+1 | \lambda_2+i, \\ p^2-1| \lambda_1-i-p\lambda_3,\ i \geq F(\lambda)} } V_I (\lambda)_{\nu_i}.
\end{equation}
\end{proposition}
Using this result, we also proved the following corollary:
\begin{corollary}[{\cite[Corollary 6.3.3]{Imai-Koskivirta-vector-bundles}}]
We have 
\[
  C_{\zip,\QQ_{\geq 0}} = \left\{  (\lambda_1,\lambda_2,\lambda_3)\in \QQ^3 
 \mid 
\lambda_1\geq \lambda_2, \ 
(p-1)\lambda_1 +\lambda_2-p\lambda_3\leq 0
  \right\}. 
\]
\end{corollary}
We now compute the ring of automorphic forms $R_{\zip}$. We first construct sections on $\GZip^\mu$ that will serve as generators for $R_{\zip}$. 

\paragraph{Partial Hasse invariants}
First, consider the weight
$\lambda=(1,0,p)$, which we denote by $\ha_{1}$. In this case, the only integer $i$ satisfying the conditions under the direct sum sign of \eqref{equ-propH0} is $i=0$. Hence $H^0(\GZip^\mu,\mathcal{V}_I (\ha_1))$ is one-dimensional and coincides with $V_I(\lambda)_{\nu_0}$. One checks immediately that the function
\begin{equation}
    \Ha_1\colon \GL_{3}\to \AA^1, \quad (a_{ij})_{1\leq i,j\leq 3} \mapsto a_{11}
\end{equation}
lies in $H^0(\GZip^\mu,\mathcal{V}_I (\ha_1))$ and hence generates this space.

Next, we choose $\lambda=(1+p,1,p)$, which we denote by $\ha_2$. In this case, the only integer $i$ satisfying the conditions under the direct sum of \eqref{equ-propH0} is $i=p$, so we deduce again that $H^0(\GZip^\mu,\mathcal{V}_I (\ha_2))$ is one-dimensional and coincides with $V_I(\lambda)_{\nu_p}$. One checks immediately that the function
\begin{equation}
    \Ha_2\colon \GL_{3}\to \AA^1, \quad (a_{ij})_{1\leq i,j\leq 3} \mapsto \left| \begin{matrix}
a_{12} & a_{13} \\
 a_{22} & a_{23}  
\end{matrix} \right|
\end{equation}
lies in $H^0(\GZip^\mu,\mathcal{V}_I (\ha_2))$. The two sections $\Ha_1$, $\Ha_2$ are called partial Hasse invariants. Their vanishing locus is the Zariski closure of one of the two codimension one flag strata in $\GF^\mu$ (see \cite[\S 5.2]{Imai-Koskivirta-partial-Hasse} for details).

\paragraph{$\mu$-ordinary Hasse invariant} The $\mu$-ordinary Hasse invariant of unitary Shimura varieties was first constructed by Goldring--Nicole in \cite{Goldring-Nicole-mu-Hasse} and by the author and T. Wedhorn in \cite{Koskivirta-Wedhorn-Hasse} using a different approach. Viewed as a function on $G_k$, the $\mu$-ordinary Hasse invariant is given by
\begin{equation}
 \Ha_\mu\left( (a_{ij})_{1\leq i,j\leq 3}\right) = a_{11}^q \Delta_1 - a_{21}^p \Delta_2 
 \quad \textrm{with } 
\begin{cases}
\Delta_1=a_{11}a_{22}-a_{12}a_{21}, \\
\Delta_2=a_{11}a_{23}-a_{21}a_{13}.
\end{cases}  
\end{equation}
The section $\Ha_\mu$ is an element of $H^0(\GZip^\mu,\Lcal(\ha_\mu))$ where $\ha_\mu = (p+1,p+1,p^2+p)$. Its non-vanishing locus is the complement of the unique open stratum of $\GZip^\mu$ (the $\mu$-ordinary locus). The main result of this section is the following:
\begin{proposition} \ \label{prop-inert-U}
\begin{assertionlist}
    \item The $k$-algebra $R_{\zip}$ is finitely generated and can be expressed as
\begin{equation}
R_{\zip}=k[\Ha_1,\Ha_2,\Ha_\mu,\deter,\deter^{-1}].
\end{equation}
\item We have $C_{\zip} =  \ZZ_{\geq 0} \ha_1 + \ZZ_{\geq 0} \ha_2 + \ZZ_{\geq 0} \ha_\mu +\ZZ \lambda_{\deter}$.
\end{assertionlist}
\end{proposition}

\begin{proof}
Note that each of the functions $\Ha_1,\Ha_2,\Ha_\mu,\deter,\deter^{-1}$ is a homogeneous element of the graded algebra $R_{\zip}$. Furthermore, for each $\lambda\in \{\ha_1,\ha_2,\ha_\mu,\lambda_{\deter}\}$, the space $H^0(\GZip^\mu,\Vcal_I(\lambda))$ is one-dimensional and coincides with a weight space $V_{I}(\lambda)_{\nu}$ of the representation $V_I(\lambda)$. In each case $\lambda =\ha_1,\ha_2,\ha_\mu,\lambda_{\deter}$, the weight $\nu$ is given respectively by $\nu_1=\ha_1$, $\nu_2=\ha_2-p\alpha_1$, $\nu_\mu=\ha_\mu$, $\nu_{\deter}=\lambda_{\deter}$. Given the form of the representation $V_I(\lambda)$, it is clear that for any $\nu\in X^*(T)$, the weight space $V_{I}(\lambda)_{\nu}$ is at most one-dimensional (however $H^0(\GZip^\mu,\Vcal_I(\lambda))$ may have higher dimension for general $\lambda$). Thus, it suffices to show that if $0\leq i\leq \lambda_1-\lambda_2$ satisfies the conditions
\begin{equation}\label{cond-i}
    p|i, \ \ p+1 | \lambda_2+i, \ \ p^2-1| \lambda_1-i-p\lambda_3, \ \ i \geq F(\lambda)
\end{equation}
then there exists non-negative integers $k_1$, $k_2$, $k_\mu$ and $k_{\deter}\in \ZZ$ such that
\begin{align*}
    \lambda &= k_1\ha_1 + k_2\ha_2+k_\mu \ha_\mu + k_{\deter}\lambda_{\deter} \\
     \nu_i &= k_1\ha_1 + k_2(\ha_2-p\alpha_1)+k_\mu \ha_\mu + k_{\deter}\lambda_{\deter}. 
\end{align*}
Since $\nu_i= \lambda-i \alpha_1$, this simply amounts to $\lambda = k_1\ha_1 + k_2\ha_2+k_\mu \ha_\mu + k_{\deter}\lambda_{\deter}$ and $i=pk_2$. We check that such a tuple $(k_1, k_2, k_\mu, k_{\deter})$ exists. Since $p|i$ by assumption, we may set $k_2\colonequals \frac{i}{p}$. We find readily:
\begin{align*}
    \lambda_1-\lambda_2 &= k_1+ pk_2 =k_1+i \\
    \lambda_2-\lambda_3 &=-pk_1-(p-1)k_2-(p^2-1)k_\mu.
\end{align*}
Therefore, we must define $k_1\colonequals \lambda_1-\lambda_2-i$ and \begin{equation}
    k_\mu \colonequals \frac{(p^2-p+1)i-p^2\lambda_1+p(p-1)\lambda_2+p\lambda_3}{p^2-1}
\end{equation}
The divisibility assumptions imply immediately that $k_{\mu}$ is an integer. Moreover, since $0\leq i\leq \lambda_1-\lambda_2$, the integers $k_1,k_2$ are non-negative. Moreover, since $i\geq F(\lambda)$, we find $k_\mu\geq 0$. Finally, by construction the element $\lambda'=\lambda - k_1\ha_1 + k_2\ha_2+k_\mu \ha_\mu$ has all three coordinates equal. Again, the divisibility assumption implies that the coordinates of $\lambda'$ are divisible by $p+1$, so we may write $\lambda'=k_{\deter}\lambda_{\deter}$. We have shown that any element in $V_I(\lambda)_{\nu_i}$ satisfying conditions \eqref{cond-i} can be expressed as a product of the functions $\Ha_1,\Ha_2,\Ha_\mu,\deter,\deter^{-1}$. By Proposition \ref{equ-propH0}, this shows the result.
\end{proof}

\bibliographystyle{test}
\bibliography{biblio_overleaf}

\newcommand{\etalchar}[1]{$^{#1}$}
\begin{thebibliography}{ABD{\etalchar{+}}66}
\providecommand{\url}[1]{\texttt{#1}}
\providecommand{\urlprefix}{URL }
\providecommand{\eprint}[2][]{\url{#2}}

\bibitem[ABD{\etalchar{+}}66]{SGA3}
M.~Artin, J.~E. Bertin, M.~Demazure, P.~Gabriel, A.~Grothendieck, M.~Raynaud
  and J.-P. Serre, S{G}{A}3: Sch\'emas en groupes, vol. 1963/64, Institut des
  Hautes \'Etudes Scientifiques, Paris, 1965/1966.

\bibitem[Bou09]{bourbaki-alg-com-5-7}
N.~Bourbaki, Alg\`ebre commutative, Springer, 2009, chapitres 5-7.

\bibitem[EvdG09]{Ekedahl-Geer-EO}
T.~Ekedahl and G.~van~der Geer, Cycle classes of the {E}-{O} stratification on
  the moduli of abelian varieties, in Algebra, arithmetic and geometry, edited
  by Y.~Tschinkel and Y.~Zarhin, vol. 269 of Progress in Math., Birkh\"auser,
  Boston, MA, June 2009, Penn. State U., PA, 2009 pp. 567--636.

\bibitem[GK18]{Goldring-Koskivirta-global-sections-compositio}
W.~Goldring and J.-S. Koskivirta, Automorphic vector bundles with global
  sections on {$G$}-{Z}ip$^{\mathcal Z}$-schemes, Compositio Math. 154 (2018),
  2586--2605, \href{https://doi.org/10.1112/S0010437X18007467}{DOI
  10.1112/S0010437X18007467}, \href{http://msp.org/idx/mr/4582533}{MR 4582533}.

\bibitem[GK19a]{Goldring-Koskivirta-Strata-Hasse}
W.~Goldring and J.-S. Koskivirta, Strata {H}asse invariants, {H}ecke algebras
  and {G}alois representations, Invent. Math. 217 (2019), no.~3, 887--984,
  \href{https://doi.org/10.1007/s00222-019-00882-5}{DOI}.

\bibitem[GK19b]{Goldring-Koskivirta-zip-flags}
W.~Goldring and J.-S. Koskivirta, Stratifications of flag spaces and
  functoriality, IMRN 2019 (2019), no.~12, 3646--3682.

\bibitem[GK22a]{Goldring-Koskivirta-divisibility}
W.~Goldring and J.-S. Koskivirta, Divisibility of mod $p$ automorphic forms and
  the cone conjecture for certain {S}himura varieties of {H}odge-type, 2022,
  preprint, \href{https://arxiv.org/abs/2211.16817}{arXiv:2211.16817}.

\bibitem[GK22b]{Goldring-Koskivirta-GS-cone}
W.~Goldring and J.-S. Koskivirta, Griffiths-{S}chmid conditions for automorphic
  forms via characteristic $p$, 2022, preprint,
  \href{https://arxiv.org/abs/2211.16819}{ arXiv:2211.16819}.

\bibitem[GN17]{Goldring-Nicole-mu-Hasse}
W.~Goldring and M.-H. Nicole, The $\mu$-ordinary {H}asse invariant of unitary
  {S}himura varieties, J. Reine Angew. Math. 728 (2017), 137--151.

\bibitem[HK00]{Hu-Keel-Mori-DS}
Y.~Hu and S.~Keel, Mori dream spaces and GIT., Michigan Mathematical Journal 48
  (2000).

\bibitem[IK]{Imai-Koskivirta-zip-schubert}
N.~Imai and J.-S. Koskivirta, Weights of mod $p$ automorphic forms and partial
  {H}asse invariants, preprint,
  \href{https://arxiv.org/abs/2211.16207}{arXiv:2211.16207}.

\bibitem[IK21a]{Imai-Koskivirta-vector-bundles}
N.~Imai and J.-S. Koskivirta, Automorphic vector bundles on the stack of
  $G$-zips, Forum Math. Sigma 9 (2021), Paper No. e37, 31 pp.

\bibitem[IK21b]{Imai-Koskivirta-partial-Hasse}
N.~Imai and J.-S. Koskivirta, Partial Hasse invariants for Shimura varieties of
  {H}odge-type, 2021, preprint, arXiv:2109.11117.

\bibitem[Jan03]{jantzen-representations}
J.~Jantzen, Representations of algebraic groups, vol. 107 of Math. Surveys and
  Monographs, American Mathematical Society, Providence, RI, 2nd edn., 2003.

\bibitem[Kis10]{Kisin-Hodge-Type-Shimura}
M.~Kisin, Integral models for {S}himura varieties of abelian type, J. Amer.
  Math. Soc. 23 (2010), no.~4, 967--1012.

\bibitem[KKLV89]{Knop-Kraft-Luna-Vust-Local-properties}
F.~Knop, H.~Kraft, D.~Luna and T.~Vust, Local {P}roperties of {A}lgebraic
  {G}roup {A}ctions, in Transformationsgruppen und Invariantentheorie, vol.~13
  of DMV Sem., Birkhauser, 1989 pp. 63--75.

\bibitem[KKV89]{Knop-Kraft-Vust-G-variety}
F.~Knop, H.~Kraft and T.~Vust, The {P}icard group of a {$G$}-variety, in
  Transformationsgruppen und Invariantentheorie, vol.~13 of DMV Sem.,
  Birkhauser, 1989 pp. 77--87.

\bibitem[Kos]{Koskivirta-Hilbert-strata}
J.-S. Koskivirta, Cohomology vanishing for Ekedahl-Oort strata on
  Hilbert-Blumenthal Shimura varieties., preprint,
  \href{https://arxiv.org/abs/2311.18562}{arXiv:2311.18562 }.

\bibitem[Kos19]{Koskivirta-automforms-GZip}
J.-S. Koskivirta, Automorphic forms on the stack of {$G$}-zips, Results Math.
  74 (2019), no.~3, Paper No. 91, 52 pp.

\bibitem[Kot92]{Kottwitz-points-shimura-varieties}
R.~Kottwitz, Points on some {S}himura varieties over finite fields., J. Amer.
  Math. Soc. 5 (1992), no.~2, 373--444.

\bibitem[KW18]{Koskivirta-Wedhorn-Hasse}
J.-S. Koskivirta and T.~Wedhorn, Generalized {$\mu$}-ordinary {H}asse
  invariants, J. Algebra 502 (2018), 98--119.

\bibitem[Moo04]{Moonen-Serre-Tate}
B.~Moonen, Serre-{T}ate theory for moduli spaces of {PEL}-type, Ann. Sci. ENS
  37 (2004), no.~2, 223--269.

\bibitem[MW04]{Moonen-Wedhorn-Discrete-Invariants}
B.~Moonen and T.~Wedhorn, Discrete invariants of varieties in positive
  characteristic, IMRN 72 (2004), 3855--3903.

\bibitem[Noe26]{Noether-fin-gen}
E.~Noether, Der Endlichkeitssatz der Invarianten endlicher linearer Gruppen der
  Charakteristik $p$, Nachr. Ges. Wiss. G\"ottingen  (1926), 28--35.

\bibitem[PWZ11]{Pink-Wedhorn-Ziegler-zip-data}
R.~Pink, T.~Wedhorn and P.~Ziegler, Algebraic zip data, Doc. Math. 16 (2011),
  253--300.

\bibitem[PWZ15]{Pink-Wedhorn-Ziegler-F-Zips-additional-structure}
R.~Pink, T.~Wedhorn and P.~Ziegler, ${F}$-zips with additional structure,
  Pacific J. Math. 274 (2015), no.~1, 183--236.

\bibitem[Spr98]{Springer-Linear-Algebraic-Groups-book}
T.~Springer, Linear Algebraic Groups, vol.~9 of Progress in Math., Birkhauser,
  2nd edn., 1998.

\bibitem[SYZ21]{Shen-Yu-Zhang-EKOR}
X.~Shen, C.-F. Yu and C.~Zhang, E{KOR} strata for {S}himura varieties with
  parahoric level structure, Duke Math. J. 170 (2021), no.~14, 3111--3236.

\bibitem[Vas99]{Vasiu-Preabelian-integral-canonical-models}
A.~Vasiu, Integral canonical models of {S}himura varieties of preabelian type,
  Asian J. Math. 3 (1999), 401--518.

\bibitem[Wor13]{Wortmann-mu-ordinary}
D.~Wortmann, The $\mu$-ordinary locus for {S}himura varieties of {H}odge type,
  2013, preprint, arXiv:1310.6444.

\bibitem[Zha18]{Zhang-EO-Hodge}
C.~Zhang, Ekedahl-{O}ort strata for good reductions of {S}himura varieties of
  {H}odge type, Canad. J. Math. 70 (2018), no.~2, 451--480.

\end{thebibliography}

\end{document}